\documentclass[11pt]{amsart}
\usepackage{amssymb,amsmath}
\def\Bbb{\mathbb}
\def\Cal{\mathcal}
\def\Dt{\partial_t}

\def\eb{\varepsilon}

\def\R {\mathbb{R}}

\def\<{\left<}
\def\>{\right>}

\def\Dx{\Delta_x}
\def\({\left(}
\def\){\right)}

\newtheorem{proposition}{Proposition}[section]
\newtheorem{theorem}[proposition]{Theorem}
\newtheorem{corollary}[proposition]{Corollary}
\newtheorem{lemma}[proposition]{Lemma}

\theoremstyle{definition}
\newtheorem{definition}[proposition]{Definition}

\newtheorem{remark}[proposition]{Remark}

\numberwithin{equation}{section}

\def \au {\rm}
\def \ti {\it}
\def \jou {\rm}

\def \no#1#2#3 {{\bf #1} (#3), #2.}
\def \eds#1#2#3 {#1, #2, #3.}

\title[Inertial manifolds for hyperbolic equations] {Inertial manifolds for the hyperbolic relaxation of semilinear parabolic equations}
\author[V. Chepyzhov, A. Kostianko and S. Zelik]{ Vladimir Chepyzhov${}^{1,2}$, Anna Kostianko${}^3$ and Sergey Zelik${}^3$}
\address{${}^1$ Institute for Information Transmission Problems\\
Russian Academy of Sciences,
Bol\textsc{\char13}shoi Karetnyi per. 19, Moscow 111994, Russia}
\address{${}^2$ National Research University Higher School of Economics
Myasnitskaya Street 20
Moscow 101000, Russia}

\address{${}^3$
University of Surrey, Department of Mathematics,
Guildford, GU2 7XH, United Kingdom.}
\email{chep@iitp.ru; a.kostianko@surrey.ac.uk; s.zelik@surrey.ac.uk}
\begin{document}
\begin{abstract} The paper gives a comprehensive study of Inertial Manifolds for  hyperbolic relaxations of an abstract semilinear parabolic equation in a Hilbert space. A new scheme of constructing  Inertial Manifolds for such type of problems is suggested and optimal spectral gap conditions which guarantee their existence are established. Moreover,  the dependence of the constructed manifolds on the relaxation parameter in the case of the parabolic singular limit is also studied.
\end{abstract}
\thanks{This work is partially supported by  the grants  14-41-00044 and 14-21-00025 of RSF as well as  grants 14-01-00346  and 15-01-03587 of RFBR}
\maketitle
\tableofcontents
\section{Introduction}\label{s0}
There is a common belief that the dynamics generated by dissipative partial differential equations (PDEs) is essentially finite-dimensional and can be effectively described by finitely many parameters which obey the associated system of ordinary differential equations (ODEs)  - the so-called Inertial Form (IF) of the initial problem. This belief is partially justified by the theory of {\it global attractors} which has been intensively developing during the last 40 years. Recall that by definition a global attractor is a compact invariant set in the phase space which attracts the images of all bounded sets as time goes to infinity. Thus, on the one hand, the attractor (if it exists) captures all the nontrivial dynamics of the system considered and, on the other hand, is essentially smaller than the initial phase space (which is, say, $L^2(\Omega)$). Moreover, under more or less general and natural assumptions one can prove that the global attractor has finite Hausdorff and box-counting dimensions, see \cite{BV,CV,MirZe,29,tem} and references therein. Due to the so-called Man\'e projection theorem, this result allows us to build up the IF with {\it H\"older} continuous non-linearity, see, say, \cite{28,Zel}.
\par
 Note that the reduction of {\it smooth} PDEs to the ODEs where the nonlinearity is {\it only} H\"older continuous does not look entirely satisfactory (e.g., even the uniqueness of solutions may be lost under such reduction) and despite many efforts building up more regular IFs under more or less general assumptions remains a mystery, see the survey \cite{Zel} and references therein. However, there is an exceptional (in a sense) case where this problem is resolved, namely, when the considered system possesses an Inertial Manifold (IM). Roughly speaking, an IM is a $C^{1}$-smooth normally hyperbolic finite-dimensional invariant submanifold of the phase space which contains the global attractor. If it exists,  the restriction of the initial equation to this manifold gives the desired IF. Of course, the existence of such an object requires some kind invariant cone or/and spectral gap conditions to be satisfied and this is a big restriction, see \cite{chep-gor,FST,mik,mal-par,rom-man, rosa-tem} and references therein. On the other hand, the recent counterexamples show that, in the case where the IM does not exist, the limit dynamics may remain infinite-dimensional (despite the fact that the global attractor has the finite box-counting dimension) and even allow us to make a conjecture that the IM is the only natural obstruction for such dynamics to exist, see \cite{EKZ,Zel} for more details. This, in particular, motivates the interest to finding the precise conditions which guarantee the existence or non-existence of IMs in various classes of dissipative systems generated by PDEs. One of the most studied nontrivial classes of such systems is given by the abstract semilinear parabolic problem, see e.g., \cite{hen},
 \begin{equation}\label{0.par}
 \Dt u+\mathcal Au=\mathcal F(u)
 \end{equation}
in a Hilbert space $\mathcal H$. Here $\Cal A$ is a positive unbounded operator which generates an analytic semigroup in $\mathcal H$ and the non-linearity $\mathcal F$ is globally Lipschitz as the map from $D(\mathcal A^{\beta})$, $0\le\beta<1$, to $\mathcal H$ with Lipschitz constant $L$. Note that although the assumption of {\it global} Lipschitz continuity of the non-linearity is usually not satisfied for the initial system, it appears naturally  after cutting off the nonlinearity outside of the absorbing ball, so this assumption is not a big restriction. Throughout  this paper, we assume implicitly that the cut-off procedure is already done and consider only globally Lipschitz continuous nonlinearities.
\par
The precise condition for \eqref{0.par} to have an IM is well-known in the case where $\mathcal A$ is self-adjoint and $\mathcal A^{-1}$ is compact:
\begin{equation}\label{0.gap}
\frac{\lambda_{N+1}-\lambda_N}{\lambda_{N+1}^\beta+\lambda_{N}^\beta}>L,
\end{equation}
where $\{\lambda_n\}_{n=1}^\infty$ are the eigenvalues of $\mathcal A$ enumerated in the non-decreasing order and $N$ is the dimension of the IM. It is also known that if \eqref{0.gap} is violated for all $N$, one can construct a smooth nonlinearity $F$ in such way that the dynamics generated by \eqref{0.par} will be infinite dimensional, see \cite{EKZ,Zel2,Zel}. However, the situation with the precise conditions become much more delicate and essentially less clear if the operator $\mathcal A$ is not selfadjoint. The most dangerous for the existence of IMs is  the appearance of Jordan cells in the spectrum of the operator $\mathcal A$. Indeed, in the model case where $\mathcal H=H\times H$, $A$ is a self-adjoint positive operator in $H$ with compact inverse, $\mathcal F$ is globally Lipschitz continuous from $\mathcal H$ to $\mathcal H$ (i.e., $\beta=0$), and
\begin{equation}
\mathcal A:=\(\begin{matrix} 1&1\\0&1\end{matrix}\)A,
\end{equation}
the precise spectral gap condition for the existence of IM reads
\begin{equation}
\frac{\lambda_{N+1}-\lambda_N}{\lambda_{N+1}^{1/2}+\lambda_{N}^{1/2}}>\sqrt{L},
\end{equation}
see \cite{Zel4},
which coincides up to the square root in the right-hand side with the case of $\beta=1/2$ in the self-adjoint case and very far from the expected condition with $\beta=0$:
\begin{equation}\label{0.sg-sharp}
\lambda_{N+1}-\lambda_N>2L.
\end{equation}
This difference causes the crucial mistake in the attempt to construct the IM for the 2D Navier-Stokes equations using the so-called Kwak transform, see \cite{Zel4,kwak,tem-wrong}. On the other hand, for the non-selfadjoint operator of the form
\begin{equation}
\mathcal A:=\(\begin{matrix} \lambda&-\omega\\\omega&\lambda\end{matrix}\)A,\ \ \lambda>0,\ \omega\in\R,
\end{equation}
which appears e.g., under the study of complex Ginzburg-Landau equations of the form
\begin{equation}\label{0.cGL}
\Dt u=(\lambda+i\omega)\Dx u-F(u,\bar u),\ \ u=u_1+i u_2,
\end{equation}
the conditions for the existence of IMs remain close to the expected \eqref{0.sg-sharp}. Moreover, the deviation from the self-adjoint case (due to the presence of $\omega\ne0$) is even helpful here. Indeed, in the case of equation \eqref{0.cGL} on 3D torus, the normally-hyperbolic IM is constructed for the case $\omega\ne0$ (see \cite{Zel3}) although as known for a long time (see \cite{rom-counter}) such an object may not exist in the case $\omega=0$.
\par
The main aim of the present paper is to give a comprehensive study of a different type deviation from the self-adjoint case, namely, the case of the so-called hyperbolic relaxation of problem \eqref{0.par}:
\begin{equation}\label{0.main}
\eb\Dt^2u+\Dt u+Au=F(u),
\end{equation}
where $\eb>0$ is the relaxation parameter, $A$ is a positive self-adjoint operator in a Hilbert space $H$ with compact inverse with the eigenvalues $\{\lambda_n\}_{n=1}^\infty$ and $F:H\to H$ is globally Lipschitz with the Lipschitz constant $L$. Introducing $v=\Dt u$, one can reduce this second order equation to the form similar to \eqref{0.par}:
\begin{equation}\label{0.bad}
\Dt\(\begin{matrix} u\\ v\end{matrix}\)+\(\begin{matrix}0&-1\\\frac1\eb A&+\frac1\eb\end{matrix}\)\(\begin{matrix}u\\v\end{matrix}\)=\(\begin{matrix}0\\\frac1\eb F(u)\end{matrix}\),
\end{equation}
however, applying the general theory is far from being straightforward here, not only since the obtained operator $\mathcal A=\mathcal A_\eb$ is not self-adjoint and even not sectorial, but also due to the strong singularity and associated boundary layers at $\eb=0$. This problem has been partially overcome in \cite{sola1,sola2}, see also \cite{chep-hyp,chep-gor05}, by introducing the specially constructed equivalent norm in the energy phase space $\mathcal E=D(A^{1/2})\times H$ which allowed to construct the IM under the following assumptions:
\begin{equation}\label{0.sg-old}
\lambda_{N+1}-\lambda_N>4L,\ \ \ \frac1\eb> 4\lambda_{N+1}+\frac{\lambda_{N+1}-\lambda_N}{R(R-1)},\ \ R:=\frac{\lambda_{N+1}-\lambda_N}{4L}>1,
\end{equation}
see \cite{chep-hyp} (see also \cite{chal,gor-cha} for more general cases including the dependence of the nonlinearity $f$ on $\Dt u$ or/and the strong damping term $\Dx\Dt u$). Unfortunately, these conditions are clearly not optimal since they do not recover the known sharp conditions \eqref{0.sg-sharp} in the parabolic limit $\eb\to0$, so the dependence of the actual required spectral gap on the parameter $\eb$ remained unclear.
\par
In the present paper, we suggest a new approach to problems of the form \eqref{0.main} which is a variation of the so-called Perron method and in a sense close to the approach of \cite{mik} (see also \cite{Zel}). Under this approach, we do not utilize the reduction \eqref{0.bad} to the first order equation and work directly with the trajectories of the second order equation \eqref{0.main} applying the Banach contraction theorem in the weighted space $L^2_{e^{\theta t}}(\R_-,H)$ where $\theta=\theta(N,\eb)$ is a properly chosen exponent. This allows us to neglect the boundary layer effects from the one hand and from the other hand, to develop a machinery for computing the crucial Lipschitz constants just by expanding the solutions of the corresponding linear problem to Fourier series  and using the Fourier transform in time for finding the optimal bounds for the norms of Fourier coefficients, see Appendix for details. In the present paper, we demonstrate this approach on the model example of problem \eqref{0.main} only although we believe that it will be helpful for many other classes of dissipative PDEs especially containing singular perturbations. We return to this somewhere else.
\par
The main result of the paper is the following theorem, see also Theorems \ref{Th1.mmain} and \ref{Th2.main1}.
\begin{theorem}\label{Th0.main} Let $A$ be positive self-adjoint operator with compact inverse in a Hilbert space $H$ with the eigenvalues $\{\lambda_n\}_{n=1}^\infty$ enumerated in the non-decreasing order and let $F:H\to H$ be globally Lipschitz with the Lipschitz constant $L$. Assume also that the numbers $L$, $\eb\ge0$ and $N\in\Bbb N$ satisfy the following conditions:
\begin{equation}\label{0.sg-sg}
\lambda_{N+1}-\lambda_N>2L,\ \ \frac1\eb\ge3\lambda_{N+1}+\lambda_N.
\end{equation}
Then equation \eqref{0.main} possesses an $N$-dimensional IM $\mathcal M=\mathcal M_\eb$ and this manifold is Lipschitz continuous with respect to $\eb$ at $\eb=0$.
\end{theorem}
We see that the first condition of \eqref{0.sg-sg} now {\it coincides} with the assumption \eqref{0.sg-sharp} for the limit parabolic equation and as shown in Section \ref{s3} below is {\it optimal} for $\eb\ne0$ as well. Thus, the optimal spectral gap condition for the perturbed problem  \eqref{0.main} is surprisingly independent of $\eb$. Concerning the second condition of \eqref{0.sg-sg} although it is essentially better than the analogous assumption of \eqref{0.sg-old}, we do not know how optimal  it is. Indeed, the natural {\it necessary} condition for the existence of the normally-hyperbolic IM here is $\frac1\eb>4\lambda_N$, so we expect that the spectral gap condition will be {\it different} from \eqref{0.sg-sharp} when
$$
4\lambda_N<\frac1\eb<3\lambda_{N+1}+\lambda_N=4\lambda_N+3(\lambda_{N+1}-\lambda_N).
$$
In particular, it is natural to expect that the allowed Lipschitz constant $L$ should tend to zero as $\frac1\eb\to4\lambda_N$. In order to avoid the technicalities, we did not present the analysis of this particular case in the paper.
\par
The paper is organized as follows. The necessary definitions and preliminary facts are given in Section \ref{s1}. The proof of the main result (Theorem \ref{Th0.main}) is presented in Section \ref{s2}. The concluding discussion concerning the optimality of the obtained spectral gap conditions, further regularity and normal hyperbolicity of the IMs, etc., is given in Section \ref{s3}. In addition, the applications of the obtained results to damped wave equations are also indicated there. Finally, all necessary estimates related with the linear problem are collected in Appendix.

\section{Preliminaries}\label{s1}
In this section, we introduce the main concepts and prepare some tools which will be used throughout the paper. We recall that our main object of study is the following hyperbolic relaxation of the abstract semilinear parabolic problem in a Hilbert space $H$:
\begin{equation}\label{1.maineq}
\eb\Dt^2 u+\Dt u+Au=F(u),\ u\big|_{t=0}=u_0,\ \ \Dt u\big|_{t=0}=u_0',
\end{equation}
where $\eb\ge0$ is a small parameter, $A:\,D(A)\to H$ is a positive self-adjoint operator in $H$ with compact inverse and the nonlinearity $F:H\to H$ is assumed to be globally Lipschitz continuous in $H$:
\begin{equation}\label{1.F}
\|F(u_1)-F(u_2)\|_H\le L\|u_1-u_2\|_H,\ \ u_1,u_2\in H,\ \ F(0)=0.
\end{equation}
Let $\{\lambda_n\}_{n=1}^\infty$ be the eigenvalues of the operator $A$ enumerated in the non-decreasing order and let $\{e_n\}_{n=1}^\infty$ be the corresponding (complete orthonormal) system of eigenvectors. Then, due to the Parseval equality, for every $u\in H$,
$$
u=\sum_{n=1}^\infty u_ne_n,\ \ u_n=(u,e_n),\ \ \|u\|^2_H=\sum_{n=1}^\infty u_n^2,
$$
where $(\cdot,\cdot)$ stands for the inner product in the Hilbert space $H$. We also denote by $H^s:=D(A^{s/2})$, $s\in\R$, the scale of Hilbert spaces generated by the operator $A$. The norms in these spaces are given by
$$
\|u\|_{H^s}^2:=\sum_{n=1}^\infty\lambda_n^s u_n^2,\ \ u=\sum_{n=1}^\infty u_ne_n.
$$
Then, obviously, $H^{s_1}\subset H^{s_2}$ for $s_1\ge s_2$ (and the embedding is compact if $s_1> s_2$) and the operator $A$ is an isometric isomorphism between $H^{s+2}$ and $H^s$. We denote by $P_N$ the orthoprojector to the subspace generated by the first $N$ eigenvectors of the operator $A$:
$$
P_Nu:=\sum_{n=1}^N(u,e_n)e_n,\ \ Q_Nu:=(1-P_N)u=\sum_{n=N+1}^\infty(u,e_n)e_n.
$$
Note that these operators act in all $H^s$, $s\in\R$, and are also  orthoprojectors in all these spaces. \par
We use the notations $\xi_u:=(u,\Dt u)$ and introduce the energy phase space $\mathcal E_\eb$ for problem \eqref{1.maineq} as follows: $\mathcal E_\eb:=H^1\times H$ if $\eb\ne0$ and $\mathcal E_0:=H^1\times H^{-1}$  endowed by the following norm:
$$
\|\xi_u\|^2_{\mathcal E_\eb}:=\eb\|\Dt u\|^2_H+\|\Dt u\|^2_{H^{-1}}+\|u\|^2_{H^1}.
$$
It worth noting that, in the limit case $\eb=0$, we need not the initial data for $\Dt u$ for the well-posedness of problem \eqref{1.maineq}, but we prefer to keep it for comparison with the trajectories corresponding to the case $\eb\ne0$ where this initial data is necessary. Of course, the limit dynamical system which corresponds to the parabolic problem \eqref{1.maineq} (with $\eb=0$) is defined not on the whole space $\mathcal E_0$, but only on the invariant manifold $\widehat{\mathcal E}_0$ determined by
\begin{equation}
\widehat{\mathcal E}_0=\big\{(u,v)\in\mathcal E_0,\  v=F(u)-Au\big\},
\end{equation}
but we will identify $\mathcal E_0$ with $\widehat{\mathcal E}_0$ everywhere in the sequel if it does not lead to misunderstandings. Analogously, the scale of energy phase spaces $\mathcal E^s_\eb$, $s\in\R$, is determined by the following norm:
$$
\|\xi_u\|^2_{\mathcal E_\eb^s}:=\eb\|\Dt u\|^2_{H^s}+\|\Dt u\|^2_{H^{s-1}}+\|u\|^2_{H^{s+1}}.
$$
The next theorem gives the standard result on the global well-posedness of problem \eqref{1.maineq}.
\begin{theorem}\label{Th1.well} Let the nonlinearity $F(u)$ satisfy \eqref{1.F} and $\eb\ge0$. Then, for every $\xi_0\in\mathcal E_\eb$, there is a unique solution $\xi_u\in C(\R_+,\mathcal E_\eb)$ of problem \eqref{1.maineq} satisfying $\xi_u\big|_{t=0}=\xi_0$ and this solution possesses the following estimate:
\begin{equation}\label{1.non-dis}
\|\xi_u(t)\|^2_{\mathcal E_\eb}+\int_{t}^{t+1}\|\Dt u(s)\|^2_H\,ds\le Ce^{K t}\|\xi_0\|^2_{\mathcal E_\eb},
\end{equation}
where the constants $C$ and $K$ depend on $L$, but are uniform with respect to $\eb\to0$. Moreover, if in addition, $\xi_0\in\mathcal E^1_\eb$, then the solution $\xi_u(t)\in\mathcal E^1_\eb$ for all $t\ge0$ and
\begin{equation}\label{1.non-dis-1}
\|\xi_u(t)\|^2_{\mathcal E^1_\eb}+\int_{t}^{t+1}\|\Dt u(s)\|^2_{H^1}\,ds\le Ce^{K t}\|\xi_0\|^2_{\mathcal E^1_\eb},
\end{equation}
and the estimate is uniform with respect to $\eb\to0$.
\end{theorem}
\begin{proof} We give below only the formal derivation of the stated estimates. Their justification can be done  using e.g., the Galerkin approximation method. The only a bit delicate place is the fact that the map $F$ is not differentiable, so the estimates involving time differentiation of $F(u)$ require some accuracy. This can be overcome (on the level of Galerkin approximations) by approximating the Lipschitz function by smooth ones without expanding the Lipschitz constant, e. g., using the mollification operator. In order to avoid the technicalities, we rest these standard details to the reader.
\par
We start with estimate \eqref{1.non-dis}. Indeed, taking the scalar product of equation \eqref{1.maineq} with $\Dt u$ and using that
$$
\|F(u)\|_{H}=\|F(u)-F(0)\|_{H}\le L\|u\|_H,
$$
we end up with
\begin{equation}
\frac12\frac d{dt}\(\eb\|\Dt u\|^2_H+\|u\|^2_{H^1}\)+\|\Dt u\|^2_H=(F(u),\Dt u)\le L\|u\|_H\|\Dt u\|_H\le \frac12\|\Dt u\|^2_H+\frac12L^2\|u\|^2_H
\end{equation}
and the Gronwall inequality together with the inequality $\|u\|^2_{H}\le \lambda_1^{-1}\|u\|^2_{H^1}$ give
\begin{equation}\label{1.en-part}
\eb\|\Dt u(t)\|^2_H+\|u(t)\|^2_{H^1}+\int_t^{t+1}\|\Dt u(s)\|^2_H\,ds\le Ce^{L^2\lambda_1^{-1}t}\(\eb\|u_0'\|^2_H+\|u_0\|^2_{H^1}\).
\end{equation}
Thus, to complete \eqref{1.non-dis}, we only need to estimate the $H^{-1}$ norm of $\Dt u$. To this end, we multiply equation \eqref{1.maineq} by $A^{-1}\Dt u$ and rewrite it in the form
$$
\frac d{dt}\|\Dt u\|^2_{H^{-1}}+\eb^{-1}\|\Dt u\|^2_{H^{-1}}\le C\eb^{-1}(\|u\|^2_{H^1}+\|F(u)\|^2_H)\le C\eb^{-1}\|u\|^2_{H^1},
$$
where $C$ depends on $L$, but is independent of $\eb\to0$. Integrating this inequality, we get the following boundary layer estimate
\begin{equation}
\|\Dt u(t)\|^2_{H^{-1}}\le \|\Dt u(0)\|^2_{H^{-1}}e^{-\eb^{-1}t}+C\eb^{-1}\int_0^te^{-\eb^{-1}(t-s)}\|u(s)\|^2_{H^1}\, ds
\end{equation}
and this estimate together with \eqref{1.en-part} give the desired estimate \eqref{1.non-dis}.
\par
To obtain the second estimate, we multiply equation \eqref{1.maineq} by $\Dt Au$ and get
\begin{equation}\label{1.dLip}
\frac d{dt}\(\frac12\eb\|\Dt u\|^2_{H^1}+\frac12\|u\|^2_{H^2}-(F(u),Au)\)+\|\Dt u\|^2_{H^1}+(F'(u)\Dt u,Au)=0.
\end{equation}
Introducing $Y(t):=\frac12\eb\|\Dt u(t)\|^2_{H^1}+\frac12\|u(t)\|^2_{H^2}-(F(u(t)),Au(t))$ and using that
$$
|(F(u),Au)|\le L\|u\|_{H^2}\|u\|_H\le L^2\|u\|^2_H+\frac14\|u\|^2_{H^2},
$$
we get
$$
\frac14\(\eb\|\Dt u\|^2_{H^1}+\|u\|^2_{H^2}\)-L\|u\|^2_H\le Y(t)\le \eb\|\Dt u\|^2_{H^1}+\|u\|^2_{H^2}+L\|u\|^2_H.
$$
This estimate, together with another obvious estimate
$$
|(F'(u)\Dt u,Au)|\le L\|\Dt u\|_{H}\|u\|_{H^2}\le \frac12\|\Dt u\|^2_{H^1}+C\|u\|^2_{H^2}
$$
allow us to transform \eqref{1.dLip} as follows
$$
\frac d{dt}Y(t)+\frac12\|\Dt u(t)\|^2_{H^1}\le C Y(t)+C\|u(t)\|^2_H.
$$
Integrating this inequality and using estimate \eqref{1.non-dis}, we arrive at
\begin{equation}\label{1.en-part1}
\eb\|\Dt u(t)\|^2_{H^1}+\|u(t)\|^2_{H^2}+\int_t^{t+1}\|\Dt u(s)\|^2_{H^1}\,ds\le Ce^{Kt}\|\xi_u(0)\|^2_{\Cal E_\eb^1}.
\end{equation}
To complete estimate \eqref{1.non-dis-1}, we only need to estimate the $H$-norm of $\Dt u$. This can be done exactly as in the derivation of \eqref{1.non-dis}, but multiplying equation \eqref{1.maineq} by $\Dt u$ instead of $A^{-1}\Dt u$. Thus, the theorem is proved.
\end{proof}
The proved theorem guarantees that the solution semigroup $S_\eb(t):\mathcal E_\eb\to\mathcal E_\eb$, $t\ge0$, is well defined by
\begin{equation}
S_\eb(t)\xi_0:=\xi_u(t),
\end{equation}
where $u(t)$ is a solution of \eqref{1.maineq} satisfying $\xi_u\big|_{t=0}=\xi_0$. Moreover, as not difficult to show, this semigroup is globally Lipschitz continuous on $\mathcal E_\eb$:
\begin{equation}\label{1.s-eq}
\|S_\eb(t)\xi_1-S_\eb(t)\xi_2\|_{\mathcal E_\eb}\le Ce^{K t}\|\xi_1-\xi_2\|_{\mathcal E_\eb},\ \ \xi_1,\xi_2\in\mathcal E_\eb,
\end{equation}
where the constants $K$ and $C$ are independent of $\eb\to0$.
\begin{remark}\label{Rem1.dis} Note that the conditions imposed on the nonlinearity $F(u)$ do not guarantee problem \eqref{1.maineq} to be dissipative. Indeed, the choice $F(u)=Lu$ is allowed and the solutions $u(t)$ may grow exponentially as $t\to\infty$. To avoid this, the extra dissipativity conditions should be added. Since the dissipativity is not essential for Inertial Manifolds (only the global Lipschitz continuity is crucial for the theory), we do not pose these conditions.
\end{remark}
We now turn to Inertial Manifolds. We start with recalling the definition adapted to our case.
\begin{definition}\label{Def1.IM} A Lipschitz continuous submanifold $\mathcal M$ of the phase space $\mathcal E_\eb$ with the base $P_NH$ (for some fixed $N\in\Bbb N$) is called an Inertial Manifold for problem \eqref{1.maineq} if
\par
1) The manifold $\mathcal M$ is strictly invariant: $S_\eb(t)\mathcal M=\mathcal M$ for $t\ge0$;
\par
2) It possesses an exponential tracking (asymptotic phase) property, namely, for any $\xi_0\in\mathcal E_\eb$ there exists $\bar\xi_0\in\mathcal M$ such that
\begin{equation}
\|S_\eb(t)\xi_0-S_\eb(t)\bar\xi_0\|_{\mathcal E_\eb}\le C\|\xi_0\|_{\mathcal E_\eb}e^{-\theta t},
\end{equation}
where the positive constants $C$ and $\theta$ are independent of $\xi_0$.
\end{definition}
We will construct the Inertial Manifold (IM) for problem \eqref{1.maineq} using the Perron method. Namely, we will use the fact that the IM is generated by all backward in time solutions $u(t)$, $t\le0$, of problem \eqref{1.maineq} which grow not too fast as $t\to-\infty$, i.e., $u\in L^2_{e^{\theta t}}(\R_-,H)$ for some properly chosen $\theta=\theta(N,\eb)$, see \cite{Zel} and Section \ref{s2} below for more details. We recall here that  (for any $V\subset\R$) the norm in the weighted space $L^2_{e^{\theta t}}(V,H)$ is defined by
\begin{equation}\label{1.norm}
\|u\|^2_{L^2_{e^{\theta t}}(V,H)}:=\int_Ve^{2\theta t}\|u(t)\|^2_{H}\,dt.
\end{equation}
In order to understand how to make a choice of $\theta$ and what initial conditions we should impose on the backward in time solutions of \eqref{1.maineq} at $t=0$, we need to investigate the {\it linear} analogue of problem \eqref{1.maineq}. This is done in details in Appendix and here we only mention the basic facts related with the structure of the spectrum of the linear problem which corresponds to $F\equiv0$. Indeed, using the Fourier expansions $u(t)=\sum_{n=1}^\infty u_n(t)e_n$, we see that the homogeneous problem \eqref{1.maineq} with $F\equiv0$ is equivalent to the following uncoupled system of ODEs
\begin{equation}\label{1.seq}
\eb u_n''(t)+u_n'(t)+\lambda_n u_n(t)=0,\ \ n\in\Bbb N,
\end{equation}
and the general solution of this problem reads
\begin{equation}\label{1.sol}
u_n(t)=p_ne^{\mu_n^+t}+q_ne^{\mu_n^-t},\ \ \mu_n^{\pm}:=\frac{-1\pm\sqrt{1-4\eb\lambda_n}}{2\eb},
\end{equation}
where $p_n,q_n\in\R$ (or $p_n,q_n\in\Bbb C$). It is easy to see that
$$
\mu_n^+\to-\lambda_n,\ \ \mu_n^-\to-\infty
$$
as $\eb\to0$, so in the limit case $\eb=0$, we just drop out the term containing $q_n$ in \eqref{1.sol}. We also see that the eigenvalues $\mu_n^+$ (resp. $\mu_n^-$) are decreasing (resp. increasing) in $n$ until they remain real. Namely, this holds for all  $n\le n_{cr}$, where $n_{cr}$ is the maximal natural $n$ satisfying $1-4\eb\lambda_n\ge0$. Therefore,
$$
\mu_{1}^-\le\cdots\le\mu_{n_{cr}}^-\le -\frac1{2\eb}\le\mu_{n_{cr}}^+\le\cdots\le\mu_{1}^+.
$$
 For $n>n_{cr}$ the eigenvalues will be complex conjugate with
$$
\operatorname{Re}\mu_n^\pm=-\frac1{2\eb}.
$$
This structure allows us to make the following observations:
\par
1. There is no hope to construct (at least the normally hyperbolic) IM diffeomorphic to $P_NH$ if $N>n_{cr}$, so we need to pose the condition like $4\eb\lambda_N<1$ to avoid this case. Actually, we will pose slightly stronger restriction that
\begin{equation}\label{1.extra}
(3\lambda_{N+1}+\lambda_N)\eb\le1,
\end{equation}
see the explanations below.
\par
2. If we want to build up the IM utilizing the spectral gap between $\lambda_N$ and $\lambda_{N+1}$, we need to fix the exponent $\theta$ satisfying
$$
-\mu_{N}^+<\theta<-\operatorname{Re}\mu_{N+1}^+.
$$
To specify this choice, we recall that in the limit case $\eb=0$, the optimal choice of this exponent is $\theta=\frac{\lambda_N+\lambda_{N+1}}2$, see \cite{Zel} for the details. The most natural generalization of this formula to the case $\eb\ne0$ would be the following one:
\begin{equation}\label{1.tht}
\theta(\eb\theta-1)+\frac{\lambda_{N}+\lambda_{N+1}}2=0, \ \ \theta=\frac{1-\sqrt{1-2\eb(\lambda_N+\lambda_{N+1})}}{2\eb}.
\end{equation}
As we will see below, this choice is indeed optimal if assumption \eqref{1.extra} holds and gives the sharp condition for the existence of the IM (note that the eigenvalues $\mu_{N+1}^\pm$ are allowed to be complex conjugate). In the case of "the last" spectral gap where
$$
4\eb\lambda_N<1,\ \ (3\lambda_{N+1}+\lambda_N)\eb>1
$$
we expect {\it different} choice of $\theta$ to be optimal, but the investigation of this case is out of scope of the paper.
\par
3. Under the assumptions \eqref{1.extra} and \eqref{1.tht}, we see that, for $n>N$ only zero solution of \eqref{1.seq} belongs to $L^2_{e^{\theta t}}(\R_-)$. In contrast to this, for $n\le N$, we have the family of such solutions parameterized by $p_n\in\R$ since
$$
e^{\mu_n^+ t}\in L^2_{e^{\theta t}}(\R_-),\ \ e^{\mu_n^-t}\notin L^2_{e^{\theta t}}(\R_-).
$$
Thus, we have $N$-dimensional family of backward solutions of
\begin{equation}\label{1.leq}
\eb\Dt^2u+\Dt u+Au=0
\end{equation}
which belong to the space $L^2_{e^{\theta t}}(\R_-,H)$ and which are parameterized by $p=(p_1,\cdots,p_N)\in P_NH$. Since
\begin{equation}\label{21.24}
\frac{u_n'(0)-\mu_n^{-}u_n(0)}{\mu_n^{+}-\mu_n^-}=\frac{\eb}{\sqrt{1-4\eb\lambda_n}}u_n'(0)+
\frac{1+\sqrt{1-4\eb\lambda_n}}{2\sqrt{1-4\eb\lambda_n}}u_n(0)=p_n,
\end{equation}
then, introducing the linear operators $\Cal P_N : H\to H$ and $\widehat{\Cal P}_N: H\to H$ via
\begin{equation}\label{1.eigenpro}
\widehat{\mathcal P}_Nu:=\sum_{n=1}^N\frac{\eb}{\sqrt{1-4\eb\lambda_n}}(u,e_n)e_n,\ \ \
\mathcal P_Nu:=\sum_{n=1}^N\frac{1+\sqrt{1-4\eb\lambda_n}}{2\sqrt{1-4\eb\lambda_n}}(u,e_n)e_n,
\end{equation}
we rewrite the initial data for problem \eqref{1.leq} in the equivalent form
$$
\widehat{\mathcal P}_N\Dt u\big|_{t=0}+\mathcal P_N u\big|_{t=0}=p.
$$
Then, we will have one-to-one correspondence between the backward solutions of \eqref{1.leq} belonging to $L^2_{e^{\theta t}}(\R_-,H)$ and elements $p\in P_NH$, see Appendix for more details. We will essentially use this observation for constructing the IM for problem \eqref{1.maineq}.

\section{Inertial Manifolds}\label{s2}
In this section, we verify the existence of the Inertial Manifold for  the  semilinear equation
\eqref{1.maineq}
in the energy phase space $\mathcal E_\eb$ and study the singular limit $\eb\to0$.
\par
We assume that $F$ is a globally  Lipschitz map with the Lipschitz constant $L$, i.e., assumption \eqref{1.F} is assumed to be satisfied.
The next theorem gives the sufficient conditions for the existence of the IM for the equation \eqref{1.maineq}.

\begin{theorem}\label{Th1.mmain} Let the function $F$ satisfy \eqref{1.F}. Assume also that for some $\eb>0$ and $N\in\mathbb N$ the following spectral gap conditions are satisfied:
\begin{equation}\label{2.sg}
\lambda_{N+1}-\lambda_N>2L,\ \ 3\lambda_{N+1}+\lambda_N\le \frac1\eb.
\end{equation}
Then equation \eqref{1.maineq} possesses an $N$-dimensional Lipschitz IM $\mathcal M=\mathcal M_\eb$ (in the sense of Definition \ref{Def1.IM}) generated by all solutions of this equation which grow backward in time slower than $e^{-\theta t}$ (namely, $u\in L^2_{e^{\theta t}}(\R_-,H)$). Here $\theta$ is the smallest root of the equation
\begin{equation}\label{2.theta}
2\theta(\eb\theta-1)+\lambda_{N+1}+\lambda_N=0.
\end{equation}
As usual, this invariant manifold is generated by the Lipschitz injective map $M:\, P_NH\to\mathcal E_\eb$ and possesses the exponential tracking property $\mathcal E_\eb$.
\end{theorem}
\begin{proof} We seek for the desired backward solutions of equation \eqref{1.maineq} as solutions of problem:
\begin{equation}\label{2.25}
\eb\Dt^2 u+\Dt u+Au=F(u),\ \ \widehat{\Cal P}_N \Dt u(0)+\Cal P_N u(0)=p\in P_NH,\ \ t\le0,
\end{equation}
which belong to the space $\Cal H_\theta^-:=L^2_{e^{\theta t}}(\R_-,H)$. Recall that the operators $\widehat {\mathcal P}_N$ and $\mathcal P_N$ are defined by \eqref{1.eigenpro}. Our first task is to verify that this problem possesses indeed a unique solution for any $p\in P_NH$ and that this solution depends in a Lipschitz continuous way on the parameter $p$. To this end, we use the Banach contraction theorem in the space $\mathcal H_\theta^-$. Namely, according to Lemma
\ref{LemA.lin}, for every $p\in P_NH$ and $h\in\mathcal H_\theta^-$, problem
\begin{equation}\label{2.lin}
\eb\Dt^2 v+\Dt v+A v=h(t),\ \ t\le 0,\ \ \widehat{\mathcal P}_N\Dt v\big|_{t=0}+\mathcal P_N v\big|_{t=0}=p,
\end{equation}
possesses a unique solution $v\in\Cal H_\theta^-$. We denote the solution (linear) operator for this problem by $\mathbb L$ (i.e., $v:=\mathbb L(h,p)$, see Lemma \ref{LemA.lin}). Then, due to \eqref{A.kest}, we have the following estimate:
\begin{equation}\label{2.kest}
\|\mathbb L(h,p)\|_{\Cal H_\theta^-}\le \frac2{\lambda_{N+1}-\lambda_N}\|h\|_{\Cal H_\theta^-}+C\|p\|_{H},
\end{equation}
where $C$ is independent of $\eb\to0$. Thus, equation \eqref{2.25} can be rewritten as a fixed point problem
\begin{equation}\label{2.fix}
u=\mathbb L(F(u),p),\ \ p\in P_NH,
\end{equation}
in the space $\Cal H_\theta^-$. We claim that the right-hand side of \eqref{2.fix} is a contraction on $\mathcal H_\theta^-$. Indeed, due to \eqref{2.kest} and \eqref{1.F},
\begin{multline}
\|\mathbb L(F(u_1),p)-\mathbb L(F(u_2),p)\|_{\Cal H_\theta^-}=\|\mathbb L(F(u_1)-F(u_2),0)\|_{\Cal H_\theta^-}\le\\\le\frac2{\lambda_{N+1}-\lambda_N}\|F(u_1)-F(u_2)\|_{\Cal H_\theta^-}\le \frac{2L}{\lambda_{N+1}-\lambda_N}\|u_1-u_2\|_{\Cal H_\theta^-}
\end{multline}
and this map is a contraction due to the first assumption of \eqref{2.sg}. Thus, equation \eqref{2.fix} and which is the same, equation \eqref{2.25}  are uniquely solvable for every $p\in P_NH$ and the solution map $U: P_NH\to\Cal H_\theta^-$ is well defined and Lipschitz continuous in $p$ (since $\mathbb L$ is Lipschitz and even linear in $p$). We are now ready to define the map $M: P_NH\to\Cal E_\eb$ which generates the desired IM via the expression
\begin{equation}\label{2.M}
M(p)=\(\begin{matrix}\Pi_1M(p)\\\Pi_2M(p)\end{matrix}\):=\(\begin{matrix} U(p)\big|_{t=0}\\\Dt U(p)\big|_{t=0}\end{matrix}\).
\end{equation}
Let us verify that this map is well-defined. Indeed, since $u=U(p)$ satisfies \eqref{2.25} and the right-hand side $h(t):=F(u(t))\in \mathcal H_\theta^-$, we have
$$
\|h\|_{\mathcal H_\theta^-}\le L\|u\|_{\mathcal H_\theta^-}\le C\|p\|_H
$$
and, due to Corollary \ref{CorA.7}, $\Dt u\in \mathcal H_\theta^-$ and
\begin{equation}\label{2.DtF}
\|\Dt h\|_{\mathcal H_\theta^-}\le L\|\Dt u\|_{\mathcal H_\theta^-}\le C\|p\|_H.
\end{equation}
Then, due to Corollary \ref{CorA.ee-reg},
\begin{equation}\label{2.im-reg}
\|M(p)\|_{\Cal E^1_\eb}=\|u(0)\|_{\mathcal E^1_\eb}\le C\(\|h\|_{W^{1,2}_{e^{\theta t}}(\R_-,H)}+\|p\|_H\)\le C\|p\|_H.
\end{equation}
Thus, the map $M: P_NH\to \mathcal E^1_\eb\subset \mathcal E_\eb$ is well-defined. Let us prove the Lipschitz continuity. Indeed, let $p_1,p_2\in P_NH$ and $u_i:=U(p_i)$, $i=1,2$, be the corresponding backward solutions. We set $\bar p:=p_1-p_2$, $\bar u:=u_1-u_2$ and $\bar h:=F(u_1)-F(u_2)$. Then, since the map $U$ is Lipschitz in $p$, we have
$$
\|\bar h\|_{\mathcal H_\theta^-}\le L\|\bar u\|_{\mathcal H_\theta^-}\le C\|\bar p\|_H.
$$
Moreover, since the function $\bar u$ solves \eqref{2.lin} (with $v$, $p$ and $h$ replaced by $\bar u$, $\bar p$ and $\bar h$ respectively), due to Corollary \ref{CorA.7}, $\bar u\in C_{e^{\theta t}}(\R_-,H)$ and, therefore, $\bar h\in L^\infty_{e^{\theta t}}(\R_-,H)$. Then, Corollary \ref{CorA.ee} gives
\begin{equation}\label{2.IM-lip}
\|M(p_1)-M(p_2)\|_{\mathcal E_\eb}=\|\bar u(0)\|_{\mathcal E_\eb}\le C\|p_1-p_2\|_H,
\end{equation}
where the constant $C$ is independent of $\eb\to0$. Thus, the well-posedness of the map $M$ and its Lipschitz continuity is verified. Moreover, by the construction of this map,
\begin{equation}
\widehat{\mathcal P}_N \Pi_2M(p)+\mathcal P_N \Pi_1M(p)=\widehat{\mathcal P}_N \Dt U(p)\big|_{t=0}+\mathcal P_N U(p)\big|_{t=0}\equiv p,\ \ \forall p\in P_NH.
\end{equation}
Therefore, the left inverse to this map exists and also uniformly Lipschitz continuous. By this reason, the set $\mathcal M:=M(P_NH)$ is a Lipschitz submanifold of $\mathcal E_\eb$, see also Remark \ref{Rem3.graph} below. The invariance of this manifold with respect to the solution semigroup generated by equation \eqref{1.maineq} follows from the definition of the map $M$.
\par
Thus, in order to verify that $\mathcal M$ is indeed an inertial manifold and finish the proof of the theorem, it is sufficient to verify the exponential tracking property. Let $\xi_u\in C(\R_+,\mathcal E_\eb)$ be an arbitrary solution of problem \eqref{1.maineq}. Following \cite{Zel}, we introduce the smooth cut-off function $\varphi(t)$ such that $\varphi(t)\equiv0$ for $t\le0$ and $\varphi(t)\equiv 1$ for $t\ge1$ and seek for the desired solution $w(t)\in\mathcal M$, $t\in\R$ in the form
$$
w(t)=\varphi(t)u(t)+v(t),
$$
where the function $v\in L^2_{e^{\theta t}}(\R,H)$. Then, $w(t)\equiv v(t)$ for $t\le0$ and the fact that $w\in\mathcal M$ will be guaranteed by the fact that $v\in L^2_{e^{\theta t}}(\R_-,H)$. On the other hand, $v(t)=w(t)-u(t)$ for $t\ge1$ and the fact that $v\in L^2_{e^{\theta t}}(\R_+,H)$ together with Corollary \ref{CorA.ee}, will imply the desired exponential tracking estimate
\begin{equation}
\|u(t)-w(t)\|_{\mathcal E_\eb}\le Ce^{-\theta t}.
\end{equation}
Thus, we only need to construct the function $v$ with the above properties. Since $w$ is also a solution of \eqref{1.maineq}, this function should satisfy the equation
\begin{equation}\label{2.exp}
\eb\Dt^2 v+\Dt v+Av=F(\varphi u+v)-\varphi F( u)-(\eb\varphi''+\varphi')u-2\eb\varphi'\Dt u:=\Phi(v,u),\ \ t\in\R.
\end{equation}
We want to apply the Banach contraction principle to this equation. To this end, we note that, for any $v\in L^2_{e^{\theta t}}(\R,H)$, the function $\Phi (v,u)\in L^2_{e^{\theta t}}(\R,H)$. Indeed, for $t\le 0$, $\Phi(v,u)=F(v)$ and
$$
\|\Phi(v,u)\|_{L^2_{e^{\theta t}}(\R_-,H)}\le \|F(v)\|_{L^2_{e^{\theta t}}(\R_-,H)}\le L\|v\|_{L^2_{e^{\theta t}}(\R,H)}
$$
(here we have implicitly used that $F(0)=0$). On the other hand, for $t\ge1$, $\Phi(v,u)=F(u+v)-F(u)$ and
$$
\|\Phi(v,u)\|_{L^2_{e^{\theta t}}(\{t\ge1\},H)}\le \|F(u+v)-F(u)\|_{L^2_{e^{\theta t}}(\{t\ge1\},H)}\le L\|v\|_{L^2_{e^{\theta t}}(\R,H)}.
$$
Finally, for $0\le t\le1$, using the fact that $u\in C(\R_+,\mathcal E_\eb)$, we see that $\Phi(v,u)\in L^2([0,1],H)$. This guarantees that $\Phi(v,u)\in L^2_{e^{\theta t}}(\R,H)$ if $v\in L^2_{e^{\theta t}}(\R,H)$. Moreover, for $v_1,v_2\in L^2_{e^{\theta t}}(\R,H)$, we have
$$
\|\Phi(v_1,u)-\Phi(v_2,u)\|_{L^2_{e^{\theta t}}(\R,H)}=\|F(\varphi u+v_1)-F(\varphi u+v_2)\|_{L^2_{e^{\theta t}}(\R,H)}\le L\|v_1-v_2\|_{L^2_{e^{\theta t}}(\R,H)}.
$$
Using now Lemma \ref{Lem1.1}, we rewrite equation \eqref{2.exp} in the equivalent form
$$
v=\mathcal L(\Phi(v,u)),\ \ v\in L^2_{e^{\theta t}}(\R,H),
$$
where the solution operator $\mathcal L$ is defined in Lemma \ref{Lem1.1}. Then, due to estimate \eqref{1.7} and the spectral gap condition \eqref{2.sg}, the function $v\to \mathcal L(\Phi(v,u))$ is a contraction on $L^2_{e^{\theta t}}(\R,H)$:
\begin{multline}
\|\mathcal L(\Phi(v_1,u)-\Phi(v_2,u))\|_{L^2_{e^{\theta t}}(\R,H)}\le\\\le \frac{2}{\lambda_{N+1}-\lambda_N}\|\Phi(v_1,u)-\Phi(v_2,u)\|_{L^2_{e^{\theta t}}(\R,H)}\le \frac{2L}{\lambda_{N+1}-\lambda_N}\|v_1-v_2\|_{L^2_{e^{\theta t}}(\R,H)}.
\end{multline}
Thus, the desired $v$ exists by Banach contraction theorem and the theorem is proved.
\end{proof}
\begin{remark}\label{Rem3.graph} Note that the map $\tilde {\mathcal P}_N:\mathcal E_\eb\to P_NH$ defined by $\tilde {\mathcal P}_N(u,v):=\hat{\Cal P}_Nv+\Cal P_Nu$ and  restricted to the subspace
$$
\mathcal H^+_N:=\{(u,v)\in P_N\mathcal E_\eb,\ \ (v,e_n)=\mu^+_n(u,e_n),\ n=1,\cdots,N\}
$$
is one-to-one. As not difficult to see, the left inverse is given by
$$
\tilde{\mathcal P}_N^{-1}p:=\(p,\sum_{n=1}^N\mu_n^+(p,e_n)e_n\)\in\mathcal H^+_N
$$
and the operator $\Bbb P_N:=\tilde{\mathcal P}_N^{-1}\circ\tilde{\mathcal P}_N:\mathcal E_\eb\to\Cal H^+_N$ is a projector. The kernel of this projector is given by
$$
\mathcal H_N^-:=\{(u,v)\in P_N\mathcal E_\eb,\ (v,e_n)=\mu_n^-(u,e_n),\ n=1,\cdots,N\}\oplus Q_N\mathcal E_\eb
$$
and the phase space $\mathcal E_\eb$ is split into a direct sum (but not orthogonal if $\eb\ne0$):
$$
\mathcal E_\eb=\mathcal H_N^+\oplus\mathcal H_N^-.
$$
Thus, analogously to the standard theory, the manifold $\mathcal M_\eb$ is a graph of the Lipschitz function $\mathbb M_\eb:\mathcal H_N^+\to\mathcal H_N^-$ given by
$$
\mathbb M_\eb(\xi_+):=\mathbb Q_NM_\eb(\tilde{\mathcal P}_N\xi_+),\ \ \xi_+\in\mathcal H_N^+,
$$
where $\mathbb Q_N:=1-\mathbb P_N$ is the projector to the space $\mathcal H_N^-$. Since all the maps $\mathbb P_N$, $\mathbb Q_N$, $\tilde{\mathcal P}_N$ are smooth as $\eb\to0$, we only need to study the dependence of the map $M_\eb$ on $\eb$.
\end{remark}
The rest of this section is devoted to the dependence of the constructed IMs on the parameter $\eb$. We are mainly interested in the parabolic singular limit $\eb\to0$. Note that all of the estimates used in the proof of Theorem \ref{Th1.mmain} are uniform with respect to $\eb\to0$ and, therefore, the result on the existence of the IM holds for the case $\eb=0$ as well. In this case, the IM $\mathcal M_0$ is generated by the solutions of the limit parabolic problem
\begin{equation}\label{2.par}
\Dt u+Au=F(u),\ \ P_Nu\big|_{t=0}=p,\ \ t\le0
\end{equation}
(since $\widehat {\Cal P}_N\to0$ and $\Cal P_N\to P_N$ as $\eb\to0$) belonging to the space $L^2_{e^{\theta t}}(\R_-,H)$, where
$$
\theta=\theta_0=\frac{\lambda_{N+1}+\lambda_N}2.
$$
The function $M(p)=M_0(p)$ is then defined by the same formula \eqref{2.M}. Obviously in this case the $\Dt u$ component of $M$ can be determined by the $u$ one using the equation:
\begin{equation}
\Pi_2M_0(p)=F(\Pi_1M_0(p))-A\Pi_1M_0(p).
\end{equation}
However, for comparison with the cases $\eb>0$, we prefer to keep both components for the limit case as well. The next theorem measures the distance between the manifolds $\mathcal M_\eb$ and $\mathcal M_0$.
\begin{theorem}\label{Th2.main1} Let the assumptions of Theorem \ref{Th1.mmain} hold. Then, the following estimate between the manifolds $\mathcal M_\eb$ and $\mathcal M_0$ (which correspond to the values of the parameter $\eb$ and $0$ respectively) is valid:
\begin{equation}\label{2.e-lip}
\|M_\eb(p)-M_0(p)\|_{\mathcal E_\eb}\le C\eb\|p\|_H,
\end{equation}
where the constant $C$ is independent of $\eb$ and $p$.
\end{theorem}
\begin{proof} Let $u_\eb:=U_\eb(p)$ and $u_0:=U_0(p)$ be the backward solutions of problems \eqref{2.25} and \eqref{2.par} respectively. Then, the difference $\hat u_\eb(t):=u_\eb(t)-u_0(t)$ solves
\begin{equation}\label{2.sing}
\begin{cases}
\eb\Dt^2\hat u_\eb+\Dt\hat u_\eb+A\hat u_\eb=[F(u_\eb)-F(u_0)]-\eb \Dt^2 u_0(t),\\ \widehat{\Cal P}_N^\eb\Dt\hat u_\eb(0)+\Cal P_N^\eb \hat u_\eb(0)=-\widehat{\Cal P}_N^\eb \Dt u_0(0)-[\Cal P_N^\eb-P_N]u_0(0):=\hat p_\eb.
\end{cases}
\end{equation}
Note that the operators $\mathcal P_N$ and $\widehat{\mathcal P}_N$ commute with the projector $P_N$ and are smooth with respect to $\eb\to0$ (see formulas \eqref{1.eigenpro}). Using this fact together with \eqref{2.im-reg} for $\eb=0$, we get
\begin{equation}\label{2.phat}
\|\hat p_\eb\|_{H}\le C\eb\|P_N\Dt u_0(0)\|_{H}+C\eb\|P_Nu_0(0)\|_H\le C\eb\|p\|_{H}.
\end{equation}
Let us now estimate the term $\Dt^2u_0$. First, due to estimate \eqref{2.DtF} with $\eb=0$, we have
$$
\|\Dt F(u_0)\|_{\Cal H_\theta^-}+\|F(u_0)\|_{\mathcal H_\theta^-}\le C\|p\|_H.
$$
Moreover, due to Corollary \ref{CorA.ee-reg}, we have
$$
\Dt u_0\in L^\infty_{e^{\theta t}}(\R_-,H)
$$
and, consequently,
$$
e^{\theta t}\|\Dt F(u_0(t))\|_H\le Le^{\theta t}\|\Dt u_0(t)\|_{H}\le C\|p\|_H.
$$
Thus, applying Corollary \ref{CorA.ee-reg} again, we arrive at
\begin{equation}\label{2.sing-n}
\|\Dt^2u_0\|_{L^2_{e^{\theta t}}(\R_-,H)}+\|\Dt^2 u_0\|_{L^\infty_{e^{\theta t}}(\R_-,H^{-1})}\le C\|p\|_H.
\end{equation}
Note also that the limit function $u_0\in \mathcal H_{\theta}^-$ with $\theta=\theta_0=(\lambda_{N+1}+\lambda_N)/2$ and the function $u_\eb$ lives in $\mathcal H_\theta^-$ with $\theta=\theta_\eb\ne\theta_0$ satisfying \eqref{2.theta}. However, according to \eqref{2.theta}
$$
\theta_\eb=\frac{\lambda_{N+1}+\lambda_N}2+\eb\theta_\eb>\frac{\lambda_{N+1}+\lambda_N}2=\theta_0
$$
and we have the uniform (with respect to $\eb\to0$) embedding
$$
 \mathcal H_{\theta_0}^-\subset\mathcal H_{\theta_\eb}^-.
$$
Thus, estimate \eqref{2.sing-n} remains valid if we replace $\Cal H_{\theta_0}^-$ by $\Cal H_{\theta_\eb}^-$ in it.
\par
We are now ready to finish the proof of the thoerem. Indeed, applying estimate \eqref{A.kest} to equation \eqref{2.sing}, we get
\begin{multline}
\|\hat u_\eb\|_{\Cal H_\theta^-}\le \frac{2}{\lambda_{N+1}-\lambda_N}\|F(u_\eb)-F(u_0)\|_{\Cal H_\theta^-}+C\eb\|\Dt^2 u_0\|_{\Cal H_\theta^-}+C\|\hat p_\eb\|_H\le\\\le \frac{2L}{\lambda_{N+1}-\lambda_N}\|\hat u_\eb\|_{\Cal H_\theta^-}+C\eb\|p\|_{H}
\end{multline}
and using the spectral gap condition \eqref{2.sg}, we arrive at
\begin{equation}
\|\hat u_\eb\|_{\Cal H_\theta^-}\le C\eb\|p\|_{H},
\end{equation}
where the constant $C$ is independent of $\eb$. Thus,
$$
\|F(u_\eb)-F(u_0)\|_{\Cal H_\theta^-}\le CL\|\hat u_\eb\|_{\Cal H_\theta^-}\le C\eb\|p\|_H
$$
and Corollary \ref{CorA.7} now implies the estimate for the $L^\infty$-norm of $\hat u_\eb$ which in turn improves the previous estimate and gives that
$$
\|F(u_\eb)-F(u_0)\|_{L^\infty_{e^{\theta t}}(\R_-,H)}\le CL\|\hat u_\eb\|_{L^\infty_{e^{\theta t}}(\R_-,H)}\le C\eb\|p\|_H.
$$
Finally, applying  Corollary \ref{CorA.ee} to equation \eqref{2.sing}, we get
$$
\|\xi_{\hat u_\eb}(0)\|_{\Cal E_\eb}\le C\eb\|p\|_{H}
$$
which gives the desired estimate \eqref{2.e-lip} and finishes the proof of the theorem.
\end{proof}
\section{Concluding remarks}\label{s3}
In this concluding section, we discuss some applications and generalizations of the proved results. We start with extra smoothness and normal hyperbolicity of the constructed IMs.
\par
\subsection{Smoothness and normal hyperbolicity} Recall that we have posed only global Lipschitz continuity assumption on the non-linearity $F$. Under this assumption we cannot expect that the IM $\mathcal M_\eb$ as well as the solution semigroup $S_\eb(t):\mathcal E_\eb\to\mathcal E_\eb$ associated with equation \eqref{1.maineq} will be more regular than Lipschitz continuous. But if the nonlinearity $F\in C^{1+\beta}(H,H)$ for some positive $\beta$, then as known  the semigroup $S_\eb(t)$ will be also $C^{1+\beta}$ with respect to the initial data. We denote its Frechet derivative at point $\xi\in\mathcal E_\eb$ by $D_\xi S_\eb(t)$. In addition, repeating word by word the proof given in \cite{Zel}, we see that the IM $\mathcal M_\eb$ is also $C^{1+\beta}$-smooth if $\beta=\beta(N,L)>0$ is small enough. The assumption $F\in C^{1+\beta}(H,H)$ may be a bit restrictive from the point of view of applications since, as known, the Nemytskii operator $u\to f(u)$ is not Frechet differentiable as an operator from $H=L^2(\Omega)$ to itself even if $f\in C_0^\infty(\R)$. This problem may be overcome in a standard way by assuming that the nonlinearity $F$ satisfies
\begin{equation}\label{3.Fr}
\|F(u_1)-F(u_2)-F'(u_1)(u_1-u_2)\|_H\le C\|u_1-u_2\|_{H^1}^\beta\|u_1-u_2\|_H,\ \ u_1,u_2\in H^1.
\end{equation}
As shown e.g., in \cite{Zel} this assumption is sufficient to obtain the $C^{1+\beta}$-smoothness of the IM. On the other hand, it allows us to overcome the problems related with the aforementioned pathological property of the Nemytskii operator.
\begin{remark}\label{Rem3.bad} Note that the $C^{1+\beta}$-regularity of the IM is guaranteed only for small positive $\beta$ and even the analyticity of $F$ does not guarantee the existence of $C^2$ smooth IM since the resounances may appear. Typically for the invariant manifolds, extra regularity of the IM requires larger spectral gaps. In particular, for the limit parabolic case $\eb=0$, we need the spectral gap like
\begin{equation}\label{3.gap}
\lambda_{N+1}-(1+\beta)\lambda_N>CL
\end{equation}
in order to have $C^{1+\beta}$ regularity of the IM, see \cite{kok,rosa-tem} for more details. Note that the assumption \eqref{3.gap} is {\it essentially} stronger than \eqref{0.sg-sharp} and is natrually satisfied only if $\lambda_n$ grow {\it exponentially} fast as $n\to\infty$. Since in applications $A$ is usually the elliptic operator in a bounded domain where such growth is impossible due to the Weyl asymptotic, one cannot expect $C^2$-smooth IMs in applications.
\end{remark}
To continue, we need to recall the concept of normal hyperbolicity adopted to our case where the phase space is infinite-dimensional and the manifold is not compact, see \cite{fen,Hirsh,rosa-tem} for more details.
\begin{definition}\label{Def3.norm} Let $S_\eb(t)\in C^{1+\beta}(\mathcal E_\eb,\mathcal E_\eb)$ and $\mathcal M_\eb$ be an $N$-dimensional $C^{1+\beta}$ submanifold of $\mathcal E_\eb$ which is inavariant with respect to the semigroup $S_\eb(t)$. Denote by $\mathcal T\mathcal M_\eb$ the tangent bundle associated with $\mathcal M_\eb$ and let $\mathcal T_\xi\mathcal M_\eb\sim \R^N$, $\xi\in\mathcal M_\eb$, be its fibers. The manifold $\mathcal M_\eb$ is called stable and absolutely normally hyperbolic if there exists a  vector bundle $\mathcal N\mathcal M_\eb$ with fibers $\mathcal N_\xi\mathcal M_\eb$ of codimension $N$ in $\mathcal E_\eb$ such that
\par
1. The bundle $\mathcal N\mathcal M_\eb$ is invariant: $D_\xi S_\eb(t)\mathcal N_\xi\mathcal M_\eb\subset \mathcal N_{S_\eb(t)\xi}\mathcal M_\eb$, $t\ge0$.
\par
2. For every $\xi\in\mathcal M_\eb$, $\mathcal E_\eb=\mathcal T_\xi\mathcal M_\eb\oplus \mathcal N_\xi\mathcal M_\eb$ and the projectors $P_\xi$ and $Q_\xi$ to the first and second components of the direct sum satisfy
\begin{equation}\label{3.direct}
\|P_\xi\|_{\mathcal L(\mathcal E_\eb,\mathcal E_\eb)}+\|Q_\xi\|_{\mathcal L(\mathcal E_\eb,\mathcal E_\eb)}\le C,
\end{equation}
where the constant $C$ is independent of $\xi\in\mathcal M_\eb$.
\par
3. There exist positive constants $C$, $\theta$ and $\kappa<\theta$ such that, for every $\xi\in\mathcal M_\eb$,
\begin{equation}\label{3.norm}
\begin{cases}
\|D_{\xi}S_\eb(t)\eta\|_{\mathcal E_\eb}\le Ce^{-(\theta +\kappa)t}\|\eta\|_{\mathcal E_\eb},\ \ \eta\in\mathcal N_\xi\mathcal M_\eb,\\
\|D_{\xi}S_\eb(t)\eta\|_{\mathcal E_\eb}\ge C^{-1}e^{-(\theta -\kappa)t}\|\eta\|_{\mathcal E_\eb},\ \ \eta\in\mathcal T_\xi\mathcal M_\eb.
\end{cases}
\end{equation}
\end{definition}
\begin{remark}\label{Rem3.NH} Since IMs are stable by definition, we adapt the definition of normal hyperbolicity to this case by excluding the unstable bundle. In the finite dimensional case, we  have the strict invariance of the stable bundle $\mathcal N\mathcal M_\eb$ which is usually not the case in infinite dimensions since the linearization $D_\xi S_\eb(t)$ may be not invertible. For instance, in the case of parabolic PDEs these operators are compact and by this reason, not invertible. Estimate \eqref{3.direct} actually follows from \eqref{3.norm} in the case when $\mathcal M_\eb$ is compact, so it is  added to treat the non-comact case. Finally, {\it absolute} normal hyperbolicity means that the exponent $\theta$ is independent of the point $\xi\in\mathcal M_\eb$. In the general definition of normal hyperbolicity this exponent may depend on the point $\xi\in\mathcal M_\eb$. We restrict ourselves to the discussion of the absolute normal hyperbolicity only by two reasons. First, the IMs constructed by the Perron method are usually absolutely normally hyperbolic (although, non-absolute normally hyperbolic IMs naturally arise when the alternative method based on the invariant cones is used, e.g., for methods involving the so-called spatial averaging, see \cite{mal-par,kost,Zel}). Second, the absolute normal hyperbolicity can be relatively easily extended to the non-compact case and the proper extension (suitable for IMs) of non-absolute hyperbolicity to the non-compact case requires the replacing of exponents in \eqref{3.norm} by more complicated functions, see \cite{fen}.
\end{remark}
\begin{theorem}\label{Th3.norm} Let the assumptions of Theorem \ref{Th1.mmain} hold and let in addition the nonlinearity $F$ satisfy \eqref{3.Fr}. Then the IM $\mathcal M_\eb$ is absolutely normally hyperbolic in the sense of Definition \ref{Def3.norm}.
\end{theorem}
The proof of this theorem is standard and we will not repeat it here, see \cite{rosa-tem} for more details.
\begin{remark}\label{Rem3.eq} Recall that according to the general theory of invariant manifolds, the IM must be normally hyperbolic in order to be robust with respect to small perturbations, see \cite{fen,Hirsh}. In addition, to the best of our knowledge, all known  more or less general schemes of constructing  IMs automatically give the normal hyperbolicity (although it is not difficult to construct the artificial examples of non-normally hyperbolic IMs, see  Theorem \ref{Th3.no} and Remark \ref{Rem3.no} below), so exactly the normally hyperbolic IMs are most interesting from the point of view of applications. On the other hand, the non-existence of a {\it normally hyperbolic} IM is usually much easier to establish than the non-existence of any IM.
In particular, the normal hyperbolicity estimates \eqref{3.norm} is relatively easy to prove or disprove looking at the equilibria $\xi_0$ of the considered semigroup. Indeed, in this case $\mathcal S_{\xi_0}(t):=D_{\xi_0}S_\eb(t)$ is a linear semigroup in $\mathcal E_\eb$ and invariant subspaces $V_+:=\mathcal T_{\xi_0}\mathcal M_\eb$ and $V_-:=\mathcal N_{\xi_0}\mathcal M_\eb$ are just the spectral subspaces which correspond to the parts of the spectrum of $\mathcal S_{\xi_0}(t)$ situated outside and inside the disk $\{|z|\le e^{-\theta t}\}$ respectively. Assume now that the spectrum of the linear operator
$$
\mathcal L_{u_0}:=\mathcal A-\mathcal F'(u_0)
$$
which corrsponds to the linearization \eqref{0.par} near the equilibrium $u_0\in\mathcal H$ is {\it discrete} and the spectral mapping theorem holds for this operator. Then enumerating its eigenvalues $\{\nu_n\}_{n=1}^\infty$ in such way that their real parts are non-increasing, we see that the $N$-dimensional normally hyperbolic IM (which must contain all equilibria by the definition) exists only if
\begin{equation}\label{3.n-spec}
0>\operatorname{Re}\nu_N>\operatorname{Re}\nu_{N+1}.
\end{equation}
The non-existence of normally hyperbolic IMs of any finite dimension is usually verified by considering several (say, two or four) equilibria and organizing the multiplicity of the associated eigenvalues in such way that for any $N\in\Bbb N$ condition \eqref{3.n-spec} fails at least at one of these equilibria, see \cite{rom-counter,sell-counter} for details.
\par
We also note that in our case of equation \eqref{0.bad} the (weak) spectral mapping theorem also holds. Indeed, it obviously holds for the unperturbed semigroup $e^{-\mathcal A t}$ and as not difficult to verify, $e^{-\mathcal L_{u_0}t}$ is a compact perturbation of $e^{-\mathcal A t}$, see \cite{semigroup}. Thus, the aforementioned scheme is applicable in our case as well and we will use it below to verify the sharpness of our spectral gap assumptions.
\end{remark}

\subsection{Sharpness of  spectral gap conditions} In this subsection, we discuss the sharpness of the proved spectral gap conditions. We start with the following simple lemma.
\begin{lemma}\label{Lem3.simple} Let $N\in\mathbb N$ be fixed. Assume also that in the case if
\begin{equation}\label{3.real}
\frac1\eb<4\lambda_N,
\end{equation}
the constant $L$ is chosen in such way that $\lambda_{N+1}-\lambda_N<2L$. Then, there exists a linear operator $F\in\mathcal L(H,H)$ such that $\|F\|_{\mathcal L(H,H)}<L$ and the linear equation
\begin{equation}\label{3.lin}
\eb\Dt^2 u+\Dt u+Au=Fu
\end{equation}
does not possess an $N$-dimensional normally hyperbolic IM.
\end{lemma}
\begin{proof} Indeed, in the case when \eqref{3.real} is violated, we may just take $F=0$ and condition \eqref{3.n-spec} will be automatically violated since $\operatorname{Re}\nu_N=\operatorname{Re}\nu_{N+1}=-\frac1{2\eb}$. Thus, we only need to consider the case where \eqref{3.real} is satisfied. Then, we define the operator $F$ via
\begin{equation}\label{3.Fsingle}
Fe_N:=-\frac{\lambda_{N+1}-\lambda_N}2 e_N,\ \ Fe_{N+1}:=+\frac{\lambda_{N+1}-\lambda_N}2 e_{N+1},\ \ Fe_n=0,\ n\ne N,\,N+1.
\end{equation}
It is not difficult to see that $\|F\|_{\mathcal L(H,H)}=\frac{\lambda_{N+1}-\lambda_N}2<L$ and, on the other hand, at zero equilibrium we have $\nu_{N}=\nu_{N+1}$ which forbid the existence of $N$ dimensional normally hyperbolic IM and finishes the proof of the lemma.
\end{proof}
We are now ready to construct an equation of the form \eqref{0.main} which does not possess any finite dimensional normally hyperbolic IM. To this end of course, the spectral gap conditions should be violated for all $N$. Namely, let $n_{cr}=n_{cr}(\eb)$ be the largest $n$ for which the inequality $\frac1\eb<4\lambda_n$ be satisfied and assume that the constant $L$ is such that
\begin{equation}\label{3.nogap}
\sup_{N\le n_{cr}}\{\lambda_{N+1}-\lambda_N\}<2L.
\end{equation}
Then, the following theorem holds.
\begin{theorem}\label{Th3.no} Let the numbers $L>0$ and $\eb\ge0$ satisfy assumption \eqref{3.nogap} and let, in addition, $L>\lambda_1$. Then, there exists a globally Lipschitz with Lipschitz constant $L$ and smooth nonlinearity $F:H\to H$ such that equation \eqref{0.main} does not possess any finite-dimensional normally hyperbolic IM.
\end{theorem}
\begin{proof} The proof follows the strategy described at Remark \ref{Rem3.eq}. We introduce two linear operators  $F^+,F^-\in\mathcal L(H,H)$ which have the following form in the  basis $\{e_n\}_{n=1}^\infty$:
$$
F^+u:=\sum_{n=1}^\infty F_n^+(u,e_n)e_n,\ \ F_{2k-1}^+:=-\frac{\lambda_{2k}-\lambda_{2k-1}}2,\ \
F^+_{2k}:=-F^+_{2k-1},\ k\in\Bbb N
$$

and
$$
F^-u:=(L-\delta)(u,e_1)e_1+\sum_{n=2}^\infty F_n^-(u,e_n)e_n,\ \ F_{2k}^-:=-\frac{\lambda_{2k+1}-\lambda_{2k}}2,\ \ F^-_{2k+1}:=-F^-_{2k},
$$
where $\delta$ is small enough to guarantee that $L-\delta>\lambda_1$.
  Finally, we replace the diagonal elements of $F^\pm_n$ by zeros for $n> n_{cr}+1$. Then, due to the condition \eqref{3.nogap}, we conclude that $\|F^\pm\|_{\mathcal L(H,H)}<L$. We now construct the globally Lipschitz non-linearity $F$ with Lipschitz constant $L$  in such way that it will possess two equilibria $u^+_0$ and $u_0^-$ such that the linearizations of \eqref{0.main} at the equilibria $u_0^\pm$ give the following equations:
\begin{equation}\label{lin-lin}
\eb\Dt^2v+\Dt v+Av=F^{\pm}v.
\end{equation}
Such $F$  exists and even may be constructed  in the diagonal form:
$$
F(u)=f_1(u_1)e_1+\sum_{n=2}^\infty f_n(u_1,u_n)e_n
$$
with the equilibria of the form $u_0^+=0$ and $u_0^-=Re_1$ and $R>0$ is a sufficiently big number.
Indeed, for any $R>0$, we may take the function $f_1(u_1)$ in the form
$$
f_1(z):=\max\left\{-\frac{\lambda_2-\lambda_1}2z,(L-\delta)(z-R)+\lambda_1R\right\}.
$$
Then, assumption $L-\lambda_1-\delta>0$ guarantees that $f_1(0)=f_1(R)-\lambda_1R=0$ and condition \eqref{3.nogap} ensures us that $|f'_1(z)|<L$. This function is only Lipschitz continuous, but applying the standard mollification operator with symmetric convolution kernel gives us the smooth analogue of $f_1$ satisfying the above properties. Thus, without loss of generality, we may assume from now on that $f_1$ is smooth.
\par
Let us construct $f_n(u_1,u_n)$ for $n>1$. To this end, we introduce the smooth cut-off function $\varphi(z)$ such that $\varphi(z)\equiv1$ for $z\le0$ and $\varphi(z)\equiv0$ for $z>\frac12$ and fix
$$
f_n(u_1,u_n)=(F_n^+\varphi(u_1/R)+F_n^-\varphi(1-u_1/R))\arctan(u_n).
$$
Then, obviously $u_0^\pm$ are the equilibria and the linearizations around them coincide with \eqref{lin-lin}. On the other hand, by construction
$$
|\partial_{u_n} f_n|<L,\ \ |\partial_{u_1}f_n|\le CR^{-1},\ \ f_n(u_1,0)=0
$$
and, therefore, we may fix the constant $R$ to be large enough to guarantee that the Lipschitz constant of the map $F(u)$ is $H$ is less than $L$, see also \cite{Zel} for the analogous construction.
\par
To conclude the proof,
 it only remains to note that by the construction of operators $F^\pm$, the eigenvalues $\nu_n(u_0^\pm)$ at these equilibria enumerated in the non-increasing order of their real parts satisfy
\begin{equation}
1.\ \operatorname{Re}\nu_n(u_0^+)=\operatorname{Re}\nu_{n+1}(u_0^+),\ \text{ $n$ is odd}\ \ \ \
2.\ \operatorname{Re}\nu_n(u_0^-)=\operatorname{Re}\nu_{n+1}(u_0^-),\ \text{ $n$ is even}.
\end{equation}
These two conditions exclude the existence of a normally hyperbolic IM of any finite dimension and finish the proof of the theorem.
\end{proof}
\begin{remark}\label{Rem3.no} Note that in the constructed example equation \eqref{0.main} is actually split to the infinite system of uncoupled ODEs
$$
\eb\Dt^2u_1+\Dt u_1+\lambda_1u_1=f_1(u_1),\ \ \eb\Dt^2u_n+\Dt u_n+\lambda_n u_n=f_n(u_1,u_n),\ \ f_n(u_1,0)=0,
$$
so it possesses a lot of IMs which are {\it not} normally hyperbolic, for instance, for any $N\in\Bbb N$, $N>1$ (such that $\lambda_{N+1}>L$), the plane $P_N\mathcal E_\eb$ will be an IM. But, in a complete agreement with the general theory, all these manifolds can be destroyed by arbitrarily small perturbations and, by this reason, are not very interesting. Indeed, if we slightly perturb the linear operator $F$ defined by \eqref{3.Fsingle} in the following way:
\begin{equation}\label{3.rot}
Fe_N:=-\frac{\lambda_{N+1}-\lambda_N}2 e_N+\delta e_{N+1},\ \ Fe_{N+1}:=+\frac{\lambda_{N+1}-\lambda_N}2 e_{N+1}-\delta e_N,
\end{equation}
where $\delta>0$ is arbitrarily small,
then the corresponding $\nu_N$ and $\nu_{N+1}$ become complex conjugate with non-zero imaginary parts and the $N$-dimensional invariant plane generated by the eigenvectors corresponding to the eigenvalues $\nu_1,\cdots,\nu_N$ in the non-perturbed case $\delta=0$ will disappear. Obviously, we also may perturb the operators $F^{\pm}$ introduced in the proof of Theorem \ref{Th3.no} in a similar way in order to destroy all aforementioned artificial non-normally hyperbolic IMs. Moreover, utilizing this idea in the spirit of \cite{EKZ}, see also \cite{Zel}, we may remove the normal hyperbolicity assumption in Theorem \ref{Th3.no} and construct the nonlinearities $F$ in such way that equation \eqref{0.main} will not possess any Lipschitz and even Log-Lipschitz inertial manifolds, see also \cite{sola1} for the non-existence of $C^1$-smooth non-normally hyperbolic inertial manifolds for damped wave equations.
\end{remark}

\subsection{Applications to damped wave equations} In this concluding subsection, we briefly discuss how to apply the obtained results to damped wave equations of the form
\begin{equation}\label{3.dw}
\eb\Dt^2 u+\Dt u-\Dx u=f(u)+g,\ \ \ u\big|_{\partial\Omega}=0,\ \ \eb>0,
\end{equation}
where $\Omega$ is a smooth bounded domain of $\R^d$, $d=1,2,3$, $A:=-\Dx$ is a Laplacian with respect to the variable $x\in\R^d$, $u=u(t,x)$ is an unknown function, $g\in H:= L^2(\Omega)$ are given external forces  and $f\in C^2(\R)$ is a given non-linearity satisfying the following dissipativity and growth restrictions:
\begin{equation}\label{3.f}
1.\ \ f(u)u\le C,\ \ \ |f''(u)|\le C(1+|u|^{q-2}),
\end{equation}
where the exponent $q\ge2$ is arbitrary if $d=1$ or $d=2$ and $q\le q_{crit}=5$ if $n=3$.
\par
Damped hyperbolic equations of the form \eqref{3.dw} are very popular model equations in the theory of attractors (which are non-trivial and interesting from both theoretic and applied points of view) and are intensively studied by many authors, see \cite{BV,CV,tem} and references therein. In particular, it is well known that at least for $q\le3$, equation \eqref{3.dw} is globally well-posed in the energy phase space $\mathcal E_\eb:=H^1\times H=H^1_0(\Omega)\times L^2(\Omega)$, generates a dissipative semigroup in it and possesses a compact global attractor $\Bbb A_\eb$ in $\mathcal E_\eb$. Moreover, this global attractor is uniformly (as $\eb\to0$) bounded in the space $\mathcal E_\eb^1$:
\begin{equation}\label{3.a-un}
\|\mathbb A_\eb\|_{\mathcal E_\eb^1}\le C,
\end{equation}
see also \cite{fab}.
The analogous result has been recently obtained for the case $q\le q_{crit}=5$ as well (under some extra technical assumptions on $f$, see \cite{sav}) based on the so-called Strichartz estimates. It is also known that for sufficiently small $\eb>0$ the analogous result holds without any restriction on the exponent $q$ in the 3D case as well, see \cite{zzel}.
\par
Note that, due to the Sobolev embedding $H^2(\Omega)\subset C(\Omega)$, estimate \eqref{3.a-un} implies that, for every trajectory $u(t)$ of equation \eqref{3.dw} belonging to the attractor,
$$
\|u(t)\|_{C(\Omega)}\le R,
$$
where $R$ is independent of the choice of $u$, $t$ and $\eb\to0$. By this reason we may cut-off the nonlinearity $f$ outside, of $|u|>2R$ by introducing the new nonlinear function $\bar f\in C^2_0(\R)$ such that
$$
\bar f(u)\equiv f(u),\ \ \text{if $|u|\le2R$}.
$$
Actually doing the cut-off procedure with a bit of accuracy, we may achieve that $\bar f$ will satisfy \eqref{3.f} with {\it exactly} the same constants as the initial $f$. This in turn will guarantee that the attractor of the modified equation
\begin{equation}\label{3.mud}
\eb\Dt^2u+\Dt u-\Dx u=\bar f(u)+g
\end{equation}
will satisfy  estimate \eqref{3.a-un} with exactly the same constant as the attractor of the initial equation. Finally, by the construction of $\bar f$, this implies that the attractor of \eqref{3.mud} {\it coincides} with the attractor $\mathbb A_\eb$ of the initial equation \eqref{3.dw}, so the cut-off procedure does not affect the attractor at all. The advantage, however, is that now $\bar f$ is globally bounded (as well as the functions $\bar f'$ and $\bar f''$) and the global Lipschitz continuity now holds, so we are able to apply the theory developed above. But, in order to satisfy the assumption $F(0)=0$, we need one more modification. Namely, we introduce the function $G=G(x)$ as a solution of the following elliptic boundary value problem
\begin{equation}\label{3.ell}
-\Dx G=\bar f(G)+g,\ \ G\big|_{\partial\Omega}=0.
\end{equation}
It is well known that under the above assumptions the solution $G$ of this problem exists and belongs to the space $H^2(\Omega)\subset C(\Omega)$ (the uniqueness is not guaranteed and usually does not hold, but we need to fix only one  of such solutions). Finally, we introduce the new independent variable $\bar u:=u-G$ and write the equation \eqref{3.mud} in the form
\begin{equation}\label{3.fin}
\eb\Dt^2\bar u+\Dt\bar u-\Dx\bar u=\bar f(\bar u+G)-\bar f(G), \ \ u\big|_{\partial\Omega}=0.
\end{equation}
Introducing now $A:=-\Dx$ with Dirichlet boundary conditions and $F(u):=\bar f(u+G)-\bar f(G)$, we see that the map $F$ is indeed globally Lipschitz as a map from $H=L^2(\Omega)$ to $H$ with the Lipschitz constant
$$
L:=\max_{u\in\R}|\bar f'(u)|<\infty
$$
and satisfies the condition $F(0)=0$. Thus, equation \eqref{3.fin} has the form of \eqref{0.main} and all of the assumptions posed on $F$ and $A$ are satisfied, therefore, to verify the existence of the IM for this problem, we only need to check the spectral gap conditions.
\par
To conclude, we discuss the possibility to find $N$ such that the spectral gap condition $\lambda_{N+1}-\lambda_N>2L$ is satisfied if the constant $L$ is given (the second assumption of \eqref{0.sg-sg} does not contain $L$ and is always satisfied (for a given $N$ if $\eb>0$ is small enough). The answer on this question strongly depends on the dimension $d$, so we discuss the cases $d=1$, $d=2$ and $d=3$ separately:
\par
1. $d=1$. In this case, due to the Weyl asymptotic, $\lambda_n\sim C_{\Omega}n^2$, so we have infinitely many spectral gaps of increasing size:
$$
\lambda_{N+1}-\lambda_N\sim c\sqrt{\lambda_N}
$$
and, for any $L$, the proper spectral gap exists. Thus, in the 1D case, the IM for problem \eqref{3.dw} always exists at least if $\eb\ge0$ is small enough.
\par
2. $d=2$. In this case, the Weyl asymptotic ($\lambda_n\sim C_{\Omega}n$) is not strong enough to guarantee the existence of spectral gaps of arbitrary size and their existence or non-existence remains a mystery. On the one hand, to the best of our knowledge, there are no examples of domains $\Omega$ without such gaps and, on the other hand, there are no results on their existence for more or less general domains. In particular, for the 2D square torus, the largest possible gaps are only logarithmic with respect to $\lambda_N$:
$$
\lambda_{N+1}-\lambda_N\sim \log\lambda_N.
$$
Thus, for general 2D domains the validity of spectral gap conditions remains an open problem.
\par
3. $d=3$. In this case, the Weyl asymptotic reads $\lambda_n\sim C_{\Omega}n^{2/3}$ and there are no reasons to expect big spectral gaps to exist unless the domain $\Omega$ is extremely symmetric. For instance, it fails even on a 3D torus and the only example known for us where  they exist is the case where $\Omega=\mathbb S^3$ is a 3D sphere (actually, these gaps exist on spheres of arbitrary dimension $d$). Thus, our applications of the IMs theory to  damped wave equations in dimension three are mainly restricted to the case where the underlying domain is a sphere.

\section*{Appendix. Key estimates for the linear equation}
\renewcommand{\thesection}{A}
\setcounter{equation}{0}
\setcounter{proposition}{0}
In this Appendix, we derive the estimates for the linear hyperbolic equation in weighted spaces which are  crucial for our construction of the inertial manifold. Namely,
let us consider the following linear damped wave equation
\begin{equation}\label{1.1}
\eb\Dt^2 u+\Dt u+A u=h(t),\ \ \eb>0,
\end{equation}
in a Hilbert space $H$. As before,  $A: D(A)\to H$ is a positive selfadjoint linear operator with compact inverse, $\lambda_1\le\lambda_2\le\cdots\le\lambda_n\le\cdots$ be the eigenvalues of $A$ and $\{e_n\}_{n=1}^\infty$ be the corresponding complete orthonormal system of eigenvectors. We also
assume that the right-hand side $h$ belongs to the weighted space
\begin{equation}\label{1.2}
\Cal H_\theta:=L^2_{e^{\theta t}}(\R,H)
\end{equation}
equipped by the norm
\begin{equation}\label{1.3}
\|v\|^2_{\Cal H_\theta}=\int_{\R}e^{2\theta t}\|v(t)\|^2_H\,dt.
\end{equation}
As not difficult to see (see e.g., the proof below), in the non-resonant case where
\begin{equation}\label{1.4}
\theta(\eb\theta-1)+\lambda_n\ne0,\ \ n\in\Bbb N,
\end{equation}
equation \eqref{1.1} is uniquely solvable in the class $u\in\Cal H_\theta$ for every $h\in \Cal H_\theta$, so the solution operator $\mathcal L: h\to u$ is well-defined. Our task is to find/estimate the norm of this operator. This is done in the following lemma.

\begin{lemma}\label{Lem1.1} Let $h\in \mathcal H_{\theta}$ and let $N\in\Bbb N$ be such that
\begin{equation}\label{1.5}
\lambda_{N+1}-\lambda_N>0,\ \ \frac1{\eb}\ge 3\lambda_{N+1}+\lambda_N.
\end{equation}
Let also the exponent $\theta\in\R$ satisfy
\begin{equation}\label{1.6}
2\theta(\eb\theta-1)+\lambda_{N+1}+\lambda_N=0.
\end{equation}
Then, problem \eqref{1.1} is uniquely solvable in the space $u\in \mathcal H_\theta$ and the solution operator $\mathcal L:\mathcal H_\theta\to\mathcal H_\theta$ ($\mathcal Lh:=u$) satisfies following estimate:
\begin{equation}\label{1.7}
\|\mathcal L\|_{\mathcal L(\Cal H_\theta,\Cal H_\theta)}\le \frac{2}{\lambda_{N+1}-\lambda_N}.
\end{equation}
\end{lemma}
\begin{proof} Changing the dependent variable $v=e^{\theta t}u$, we have
\begin{equation}\label{1.8}
\eb \Dt^2v+(1-2\eb\theta)\Dt v+Av+\theta(\eb\theta-1)v=\tilde h(t):=e^{\theta t}h(t),
\end{equation}
so the problem is reduced to the analogous non-weighted estimate for the equivalent equation \eqref{1.8}. At the next step, we split the solution $v(t)$ into the Fourier series
\begin{equation}\label{1.9}
v(t)=\sum_{n=1}^\infty y_n(t)e_n,\ \ \ \tilde h(t)=\sum_{n=1}^\infty h_n(t)e_n.
\end{equation}
Then, equation \eqref{1.8} reads
\begin{equation}\label{1.10}
\eb y_n''(t)+(1-2\eb\theta)y_n'(t)+(\lambda_n+\theta(\eb\theta-1))y_n(t)=h_n(t),\ n\in\mathbb N
\end{equation}
and the desired norm can be found by
\begin{equation}\label{1.11}
\|\Cal L\|_{\Cal L(\Cal H_\theta,\Cal H_\theta)}=\max_{n\in\Bbb N}\|\Cal L_n\|_{\Cal L(L^2(\R),L^2(\R))},
\end{equation}
where $\Cal L_n$ are the solution operators for equations \eqref{1.10}.
\par
To compute these norms, we use the Fourier transform and the Plancherel theorem. Indeed,
\begin{equation}\label{1.12}
\hat y_n(\mu)=R_n(\mu)^{-1}\hat h_n(\mu),\ \ R_n(\mu):=-\eb^2\mu^2+i(1-2\eb\theta)\mu+\lambda_n+\theta(\eb\theta-1)
\end{equation}
and, therefore,
\begin{equation}\label{1.13}
\|\Cal L_n\|_{\Cal L(L^2(\R),L^2(\R))}=\max_{\mu\in\R}|R_n(\mu)^{-1}|=\frac1{\min_{\mu\in\R}|R_n(\mu)|}.
\end{equation}
Thus, we only need to prove that, under the above assumptions,
\begin{equation}\label{1.14}
\min_{\mu\in\R}|R_n(\mu)|\ge \frac{\lambda_{N+1}-\lambda_N}2,
\end{equation}
for all $n\in\Bbb N$.
As not difficult to compute,
\begin{multline}\label{1.15}
|R_n(\mu)|^2=(-\eb\mu^2+\lambda_n+\theta(\eb\theta-1))^2+\mu^2(1-2\eb\theta)^2=\\=
\eb^2\mu^4+(1-2\eb\lambda_n+2\eb\theta(\eb\theta-1))\mu^2+(\lambda_n+\theta(\eb\theta-1))^2
\end{multline}
and it remains to minimize the quadratic polynomial
\begin{equation}\label{1.16}
Q(z):=\eb z^2+(1-2\eb\lambda_n+2\eb\theta(\eb\theta-1))z+(\lambda_n+\theta(\eb\theta-1))^2
\end{equation}
on the semiaxis $z\ge0$. There are two possibilities:
\par
{\it Case I.} $(1-2\eb\lambda_n+2\eb\theta(\eb\theta-1))\ge0$. Then, the minimum is achieved at $z=0$ and is equal to
\begin{equation}\label{1.17}
Q(0)=(\lambda_n+\theta(\eb\theta-1))^2=\(\lambda_n-\frac{\lambda_{N+1}+\lambda_N}2\)^2.
\end{equation}
The minimum of the function $n\to \(\lambda_n-\frac{\lambda_{N+1}+\lambda_N}2\)^2$ is achieved at two points $n=N$ and $n=N+1$ and is equal to $(\lambda_{N+1}-\lambda_N)^2/4$. Therefore, for this case the lemma is proved.
\par
{\it Case II.} $(1-2\eb\lambda_n+2\eb\theta(\eb\theta-1))<0$. In this case, the minimum is achieved at the vertex of the parabola and is equal to
\begin{multline}
Q_{min}=(\lambda_n+\theta(\eb\theta-1))^2-\frac1{(2\eb)^2}(1-2\eb\lambda_n+2\eb\theta(\eb\theta-1))^2=
\\=\(\frac1{2\eb}+2\theta(\eb\theta-1)\)\(2\lambda_n-\frac1{2\eb}\).
\end{multline}
Using that $2\lambda_n+\lambda_{N+1}+\lambda_N>\frac1\eb$ in Case II, we get
\begin{equation}\label{1.18}
Q_{min}=\(\frac1{2\eb}-\lambda_N-\lambda_{N+1}\)\(2\lambda_n-\frac1{2\eb}\)\ge \(\frac1{2\eb}-\lambda_N-\lambda_{N+1}\)^2,
\end{equation}
where we have used that the first multiplier is non-negative due to assumption \eqref{1.5} (here we only need that $\theta$ is real and \eqref{1.5} is not used in full strength). Moreover, due to \eqref{1.5},
$$
\frac1{2\eb}\ge \lambda_{N+1}+\frac{\lambda_{N+1}+\lambda_N}2
$$
which gives
$$
 \frac1{2\eb}-\lambda_N-\lambda_{N+1}\ge (\lambda_{N+1}-\lambda_N)/2
 $$
 and finishes the proof of the lemma.
\end{proof}
The next lemma gives the extra smoothing properties of the map $\mathcal L$.
\begin{lemma}\label{Cor1.1} Let the assumptions of Lemma \ref{Lem1.1} hold and $h\in\mathcal H_\theta$ and $u:=\mathcal Lh$. Then,
$$
(u,\Dt u)\in C_{e^{\theta t}}(\R,\mathcal E), \ \eb A^{-1/2}\Dt^2 u\in\mathcal H_\theta
$$
and the following estimate holds:
\begin{multline}\label{1.sm}
\eb\|\Dt u\|_{C_{e^{\theta t}}(\R,H)}^2+\|u\|_{C_{e^{\theta t}}(\R,H^1)}^2+\|u\|^2_{L^2_{e^{\theta t}}(\R,H^1)}+\\+\|\Dt u\|_{L^2_{e^{\theta t}}(\R,H)}^2+\eb^2\|\Dt^2 u\|^2_{L^2_{e^{\theta t}}(\R,H^{-1})}\le C_N\|h\|_{L^2_{e^{\theta t}}(\R,H)}^2,
\end{multline}
where the constant $C_N$ depends on $N$, but is independent of $\eb\to0$ and $h$.
\end{lemma}
\begin{proof} Instead of estimating the solution $u$ of \eqref{1.1} in weighted spaces, it is equivalent to estimate the solution $v$ of problem \eqref{1.8} in the non-weighted spaces. We also remind that $(1-2\theta\eb)=\sqrt{1-2\eb(\lambda_N+\lambda_{N+1})}>0$, so multiplying equation \eqref{1.8} by $\Dt v$, we get
\begin{multline}\label{2.dif}
\frac12\frac d{dt}\(\eb\|\Dt v\|^2_{H}+\|v\|^2_{H^1}\)+\alpha\|\Dt v\|^2_{H}\le\\\le |(\tilde h,\Dt v)|+\frac{\lambda_N+\lambda_{N+1}}2|(v,\Dt v)|\le \frac\alpha2\|\Dt v\|^2_{H}+C_N\(\|\tilde h\|^2_H+\|v\|^2_{H}\)
\end{multline}
for some positive $\alpha$.
Integrating this estimate over time interval $(-\infty,t)$ and using that $v$ vanishes at $-\infty$, we arrive at
\begin{multline}\label{2.int1}
\eb\|\Dt v(t)\|^2_H+\|v(t)\|^2_{H^1}+\alpha\int_{-\infty}^t\|\Dt v(s)\|^2_H\,ds\le\\\le C_N\int_{-\infty}^t\|\tilde h(s)\|^2_{H}\,ds+C_N\int_{-\infty}^t\|v(s)\|^2_{H}\,ds\le \bar C_N\|h\|^2_{\mathcal H_\theta},
\end{multline}
where we have used \eqref{1.7} in order to estimate $v$ in the RHS. To estimate the $L^2$-norm of $\|u(t)\|^2_{H^1}$, we multiply equation \eqref{1.8} by $v(t)$ to get
\begin{equation}
\frac d{dt}\(\eb(v,\Dt v)+\frac{1-2\eb\theta}2\|v\|^2_{H}\)+\|v\|^2_{H^1}=\eb\|\Dt v\|^2_H+\theta(1-\eb\theta)\|v\|^2_H+(\tilde h,v).
\end{equation}
 Integrating this equality over $t\in\R$ and using already established parts of \eqref{1.sm} )obtained in \eqref{2.int1}, we end up with
 $$
 \|u\|^2_{L^2_{e^{\theta t}}(\R,H^1)}\le C\|h\|^2_{L^2_{e^{\theta t}}(\R,H)}.
 $$
 Thus, it only remains to estimate the norm of the second derivative. This follows just by expressing the term $\eb\Dt^2 v$ from \eqref{1.8} and estimating the RHS using the already proved estimate \eqref{2.int1} and the lemma is proved.
\end{proof}
\begin{remark}\label{RemA.1} Arguing exactly as in the proof of Lemma \ref{Cor1.1}, we may obtain the analogous estimate for the solutions $u(t)$ defined on a semiaxis  $\R_-$ only if the estimate for the $L^2_{e^{\theta t}}(\R_-,H)$-norm of $u$ is known. Namely, the following estimate holds:
\begin{multline}\label{1.sm-min}
\eb\|\Dt u\|_{C_{e^{\theta t}}(\R_-,H)}^2+\|u\|_{C_{e^{\theta t}}(\R_-,H^1)}^2+\|u\|^2_{L^2_{e^{\theta t}}(\R_-,H^1)}+\\+\|\Dt u\|_{L^2_{e^{\theta t}}(\R_-,H)}^2+\eb^2\|\Dt^2 u\|^2_{L^2_{e^{\theta t}}(\R_-,H^{-1})}\le C_N(\|h\|_{L^2_{e^{\theta t}}(\R_-,H)}^2+\|u\|^2_{L^2_{e^{\theta t}}(\R_-,H)}).
\end{multline}
Crucial for us here is the fact that this estimate does not  depend explicitly  on the initial data for the backward solution $u$ at $t=0$.
\end{remark}
We now consider the solutions $u(t)$ of equation \eqref{1.1} defined on the negative semiaxis $t\le0$ only and belonging to the space $L^2_{e^{\theta t}}(\R_-,H)$. We start with the case $h\equiv0$.

\begin{lemma}\label{LemA1.s} Let the above assumptions hold. Then, for any $p\in P_N H\sim\R^N$, the following problem:
\begin{equation}\label{1.semi}
\eb\Dt^2 u+\Dt u+Au=0,\ \ P_Nu\big|_{t=0}=p,
\end{equation}
possesses a unique solution $u\in L^2_{e^{\theta t}}(\R_-,H)$ and this solution is given by the following expression:
\begin{equation}\label{1.S}
u(t):=(\mathcal Sp)(t)=\sum_{n=1}^Ne^{\mu^+_n t}(p, e_n)e_n,
\end{equation}
where $\mu^+_n:=\frac{-1+\sqrt{1-4\eb\lambda_n}}{2\eb}$.
\end{lemma}
\begin{proof}Indeed, in the Fourier basis equation \eqref{1.semi} reads
\begin{equation}\label{2.22}
\eb u_n''(t)+u_n'(t)+\lambda_n u_n(t)=0,\ \ \ n\in\Bbb N,
\end{equation}
and its  general solution  is given by
\begin{equation}\label{2.23}
u_n(t)=p_ne^{\mu_n^+ t}+q_n e^{\mu_n^- t},\ \ \mu_n^{\pm}:=\frac{-1\pm\sqrt{1-4\eb\lambda_n}}{2\eb}.
\end{equation}
It is not difficult to see that for $n\ge N+1$ both exponents $e^{\mu_n^{\pm}t}$ grow faster than $e^{-\theta t}$ as $t\to-\infty$, so $u_n(t)=0$ is the unique solution of \eqref{2.22} satisfying the desired property. For $n\le N$, we have $-\mu_n^{-}<\theta<-\mu_n^+$. Thus, there is a one-parameter family of desired solutions of \eqref{2.22} given by \eqref{1.S} which is uniquely determined by the initial condition $P_Nu\big|_{t=0}=p$ and the lemma is proved.
\end{proof}
We now reformulate the initial condition $P_Nu(0)=p$ in the form which allows us to study the non-homogeneous equations as well. Namely, as follows from \eqref{2.23}, see also \eqref{21.24},
\begin{equation}\label{2.24}
\frac{u_n'(0)-\mu_n^{-}u_n(0)}{\mu_n^{+}-\mu_n^-}=\frac{\eb}{\sqrt{1-4\eb\lambda_n}}u_n'(0)+
\frac{1+\sqrt{1-4\eb\lambda_n}}{2\sqrt{1-4\eb\lambda_n}}u_n(0)=p_n.
\end{equation}
Thus, introducing the linear operators $\Cal P_N : H\to H$ and $\widehat{\Cal P}_N: H\to H$ via \eqref{1.eigenpro},
we rewrite the initial data for problem \eqref{1.semi} in the equivalent form
$$
\widehat{\mathcal P}_N\Dt u\big|_{t=0}+\mathcal P_N u\big|_{t=0}=p.
$$
We now turn to the non-homogeneous version of problem \eqref{1.semi}.
\begin{lemma}\label{LemA.lin} Let the above assumptions hold. Then, for every $p\in P_NH$ and every $h\in L^2_{e^{\theta t}}(\R_-,H))$, problem
\begin{equation}\label{A.main-lin}
\eb\Dt^2 u+\Dt u+A u=h(t),\ \ t\le 0,\ \ \widehat{\mathcal P}_N\Dt u\big|_{t=0}+\mathcal P_N u\big|_{t=0}=p
\end{equation}
possesses a unique solution $u\in L^2_{e^{\theta t}}(\R_-,H)$. This solution can be written in the form
\begin{equation}\label{A.main}
u=\mathcal Sp+\mathcal L h,
\end{equation}
where $\mathcal L$ is defined in Lemma \ref{Lem1.1} and $h$ is extended by zero for positive values of $t$. Moreover, the following estimate holds:
\begin{equation}\label{A.kest}
\|u\|_{L^2_{e^{\theta t}}(\R_-,H)}\le \frac2{\lambda_{N+1}-\lambda_N}\|h\|_{L^2_{e^{\theta t}}(\R_-,H)}+C\|p\|_{H},
\end{equation}
where the constant $C$ may depend on $N$, but is independent of $h$, $p$ and $\eb$.
\end{lemma}

\begin{proof} Indeed, let $\tilde u:=\mathcal L h$. Then, since $h$ is extended by zero for $t\ge0$, this function solves
$$
\eb\Dt^2 u+\Dt u+Au=0,\ \ t\ge0.
$$
Moreover, $\tilde u\in L^2_{e^{\theta t}}(\R,H)\subset L^2_{e^{\theta t}}(\R_+,H)$. The Fourier components $\tilde u_n(t)$ have the form \eqref{2.23} at least for $n\le N$ and in order to belong to the space $L^2_{e^{\theta t}}(\R_+,H)$, they should satisfy   $p_n=0$ for all $n\le N$. Therefore, by the definition of the operators $\widehat{\mathcal P}_N$ and $\mathcal P_N$, we have
$$
\widehat {\mathcal P}_N\Dt \tilde u(0)+\mathcal P_N\tilde u(0)=0.
$$
Thus, the difference $\bar u:=u-\tilde u$ satisfies
$$
\eb\Dt^2\bar u+\Dt\bar u+A\bar u=0,\ \ \widehat {\mathcal P}_N\Dt \bar u(0)+\mathcal P_N\bar u(0)=p
$$
and by Lemma \ref{LemA1.s}, $\bar u=\mathcal Sp$. This gives the unique solvability of problem \eqref{A.main-lin} as well as formula \eqref{A.main}. The key estimate \eqref{A.kest} follows now from Lemma \ref{Lem1.1} and the elementary fact that
\begin{equation}\label{A-S}
\|\mathcal S\|_{\mathcal L(H,L^2_{e^{\theta t}}(\R_-,H))}\le C,
\end{equation}
where the constant $C$ is independent of $\eb$ and the lemma is proved.
\end{proof}
The next corollary gives the extra smoothness analogously to Lemma \ref{Cor1.1}.

\begin{corollary}\label{CorA.7} Let the assumptions of the previous lemma hold. Then, the solution $u$ satisfies
\begin{multline}\label{A.est}
\eb\|\Dt u\|_{C_{e^{\theta t}}(\R_-,H)}^2+\|u\|_{C_{e^{\theta t}}(\R_-,H^1)}^2+\|u\|^2_{L^2_{e^{\theta t}}(\R_-,H^1)}+\\+\|\Dt u\|_{L^2_{e^{\theta t}}(\R_-,H)}^2+\eb^2\|\Dt^2 u\|^2_{L^2_{e^{\theta t}}(\R_-,H^{-1})}\le C_N(\|h\|_{L^2_{e^{\theta t}}(\R_-,H)}^2+\|p\|_{H}^2),
\end{multline}
where the constant $C_N$ depends on $N$, but is independent of $\eb\to0$ $p$ and $h$.
\end{corollary}
\begin{proof} Indeed, this is an immediate corollary of estimates \eqref{A.kest} and \eqref{1.sm-min}.
\end{proof}
At the next step, we recall that the norm in our energy phase space $\mathcal E_\eb$ is given by
$$
\|\xi_u\|_{\mathcal E_\eb}^2:=\eb\|\Dt u\|^2_H+\|\Dt u\|^2_{H^{-1}}+\|u\|^2_{H^1},\
 \ \xi_u:=(u,\Dt u).
$$
The corollary below gives the uniform estimate for this norm for the solutions of \eqref{A.main-lin} under the extra assumptions on the right-hand side $h$.
\begin{corollary}\label{CorA.ee} Let the assumptions of Lemma \ref{LemA.lin} hold and let, in addition, the function $h\in L^\infty_{e^{\theta t}}(\R_-,H^{-1})$. Then, the following estimate holds:
\begin{equation}\label{A.phaseest}
\|\xi_u(t)\|^2_{\mathcal E_\eb}\le Ce^{-2\theta t}\(\|h\|^2_{L^2_{e^{\theta t}}(\R_-,H)}+\|h\|^2_{L^\infty_{e^{\theta t}}(\R_-,H^{-1})}+\|p\|^2_H\),
\end{equation}
where the constant $C$ depends on $N$, but is independent of $\eb\to0$.
\end{corollary}
\begin{proof} The desired estimate for the terms $\eb\|\Dt u(t)\|^2_H+\|u(t)\|^2_{H^1}$ is obtained in \eqref{A.est}, so we only need to estimate the term $\|\Dt u(t)\|^2_{H^{-1}}$. To this end, we multiply equation \eqref{A.main-lin} by $A^{-1}\Dt u$ and get
\begin{equation}
\eb\frac d{dt}\|\Dt u\|^2_{H^{-1}}+2\|\Dt u\|^2_{H^{-1}}=-2(u,\Dt u)+2(h,A^{-1}\Dt u)\le \|\Dt u\|^2_{H^{-1}}+2(\|u\|^2_{H^1}+\|h\|^2_{H^{-1}}).
\end{equation}
Integrating this inequality over $[t-1,t]$, we arrive at the following boundary layer type estimate:
\begin{equation}
\|\Dt u(t)\|^2_{H^{-1}}\le \eb^{-1}e^{-\frac1\eb}\|\Dt u(t-1)\|^2_{H^{-1}}+2\eb^{-1}\int_{t-1}^te^{-\eb^{-1}(t-s)}(\|h(s)\|^2_{H^{-1}}+\|u(s)\|^2_{H^{1}})\,ds
\end{equation}
which finally gives us that
\begin{equation}
\|\Dt u(t)\|^2_{H^{-1}}\le C\eb\|\Dt u(t-1)\|^2_{H}+C(\|h\|_{L^\infty(t-1,t;H^{-1})}^2+\|u\|^2_{L^\infty(t-1,t;H^1)}).
\end{equation}
This estimate together with \eqref{A.est} gives the desired control for the $\|\Dt u(t)\|_{H^{-1}}$ and finishes the proof of the corollary.
\end{proof}
We conclude the Appendix by stating the analogous estimates for the case where the right-hand side $h$ is more regular in time.
\begin{corollary}\label{CorA.reg} Let the assumptions of Lemma \ref{LemA.lin} hold and let $\Dt h\in L^2_{e^{\theta t}}(\R_-,H)$. Then the following estimate is valid for the solution $u(t)$:
 \begin{multline}\label{A.est-reg}
\eb\|\Dt^2 u\|_{C_{e^{\theta t}}(\R_-,H)}^2+\|\Dt u\|_{C_{e^{\theta t}}(\R_-,H^1)}^2+\|\Dt u\|^2_{L^2_{e^{\theta t}}(\R_-,H^1)}+\|u\|^2_{L^2_{e^{\theta t}}(\R_-,H^2)}\\+\|u\|^2_{C_{e^{\theta t}}(\R_-,H^2)}+\|\Dt^2 u\|_{L^2_{e^{\theta t}}(\R_-,H)}^2+\eb^2\|\Dt^3 u\|^2_{L^2_{e^{\theta t}}(\R_-,H^{-1})}\le C_N(\|h\|_{W^{1,2}_{e^{\theta t}}(\R_-,H)}^2+\|p\|_{H}^2),
\end{multline}
where the constant $C_N$ depends on $N$, but is independent of $\eb\to0$, $p$ and $h$.
\end{corollary}
\begin{proof} Indeed, differentiating \eqref{A.main-lin} in time and denoting $v(t):=\Dt u(t)$, we arrive at the equation
\begin{equation}\label{A.diff}
\eb\Dt^2v+\Dt v+Av=\Dt h,\ \ \ t\le0,
\end{equation}
which is again of the form of \eqref{A.main-lin}. Moreover, due to \eqref{A.est}, we have the control of the expression $\|v\|_{L^2_{e^{\theta t}}(\R_-,H)}^2$. Therefore, we get all parts of estimate \eqref{A.est-reg} from estimate \eqref{1.sm-min} applied to equation \eqref{A.diff}. Expressing now the term $Au$ from equation \eqref{A.main-lin} and using the already obtained parts of \eqref{A.est-reg} for estimating the term $\eb\Dt^2u=\eb\Dt v$, we get the desired estimate for the $H^2$-norms of $u$ and finish the proof of the corollary.
\end{proof}
The next corollary is the analogue of Corollary \ref{CorA.ee} for this more regular case. To state it, we first recall that the second energy norm in the phase space is given by
$$
\|\xi_u\|^2_{\mathcal E^1_\eb}:=\eb\|\Dt u\|^2_{H^1}+\|u\|^2_{H^2}+\|\Dt u\|^2_{H}.
$$
\begin{corollary}\label{CorA.ee-reg} Let the assumptions of Corollary \ref{CorA.reg} hold. Then, the solution $u(t)$ of problem \eqref{A.main-lin} satisfies the following estimate:
\begin{equation}\label{A.phaseest1}
\|\xi_u(t)\|^2_{\mathcal E_\eb^1}\le Ce^{-2\theta t}\(\|h\|^2_{W^{1,2}_{e^{\theta t}}(\R_-,H)}+\|p\|^2_H\),
\end{equation}
where the constant $C$ depends on $N$, but is independent of $\eb\to0$. Moreover, if in addition $\Dt h\in L^\infty_{e^{\theta t}}(\R_-,H^{-1})$, then the following estimate holds:
\begin{equation}\label{A.phaseest2}
\|\xi_{\Dt u}(t)\|^2_{\mathcal E_\eb}\le Ce^{-2\theta t}\(\|h\|^2_{W^{1,2}_{e^{\theta t}}(\R_-,H)}+\|\Dt h\|^2_{L^\infty_{e^{\theta t}}(\R_-,H^{-1})}+\|p\|^2_H\),
\end{equation}
where the constant $C$ depends on $N$, but is independent of $\eb\to0$.
\end{corollary}
Indeed, estimate \eqref{A.phaseest2} follows from \eqref{A.phaseest} applied to the equation \eqref{A.diff} and \eqref{A.phaseest1} is already obtained in \eqref{A.est-reg}.


\end{document}